\documentclass[11pt,a4paper]{article}
\usepackage[margin=2.2cm]{geometry}

\usepackage{srcltx,graphicx}
\usepackage{amsmath, amssymb, amsthm}
\usepackage{bm}
\usepackage{color, xcolor}
\usepackage{array}
\usepackage{subfig}
\usepackage{lscape}
\usepackage{algorithm}      
\usepackage[noend]{algpseudocode}
\usepackage{multirow}
\usepackage{psfrag}
\usepackage{pifont}
\usepackage{bbding}
\usepackage{comment}
\usepackage{booktabs}
\usepackage{hyperref}
\hypersetup{
    colorlinks,
    linkcolor = {blue!50!black},
    citecolor = {blue!50!black},
    urlcolor  = {blue!80!black}
}
\usepackage{nccmath}
\usepackage{authblk}
\usepackage{graphicx}
\newtheorem{remark}{Remark}
\newtheorem{theorem}{Theorem}
\newtheorem{lemma}{Lemma}

\usepackage[numbers]{natbib}

\usepackage{mathptmx}

\begin{document}
\newpage
\title{Asymptotic Analysis of Shallow Water Moment Equations}
\author[1]{Mieke Daemen}
\author[1]{Julio Careaga}
\author[2]{Zhenning Cai}
\author[1,3]{Julian Koellermeier}

\affil[1]{
    Bernoulli Institute, University of Groningen,
    Nijenborgh 9, Groningen, 9747 AG, The Netherlands
}

\affil[2]{
    Department of Mathematics, National University of Singapore,
    10 Lower Kent Ridge Road, Singapore 119076, Singapore
}

\affil[3]{
    Department of Mathematics, Computer Science and Statistics, Ghent University,
    Krijgslaan 281-S9, Ghent, 9000, Belgium
}
\date{}

\maketitle

\begin{abstract}
The Shallow Water Moment Equations (SWME) are an extension of the Shallow Water Equations (SWE) for improved modelling of free-surface flows. In contrast to the SWE, the SWME incorporate vertical velocity profile information. The SWME framework approximates vertical velocity profiles using a polynomial expansion with Legendre polynomials and polynomial coefficients, also called moment variables.
The SWME have an increased number of variables that must always be incorporated, even when the flow approaches a viscous slip equilibrium state that could be characterised by vanishing moment variables. To reduce the complexity of the SWME in cases proximate to this equilibrium, we conduct an asymptotic analysis of the SWME. This yields the closed form Reduced Shallow Water Moment Equations (RSWME) for deviations from the equilibrium. The RSWME have fewer variables, compared to the SWME. The hyperbolicity of the RSWME is analysed. Numerical tests include a wave with a sharp height gradient, a smoother height gradient and a square root velocity profile. The numerical tests demonstrate that the RSWME reduce computational cost up to 77\% compared to the SWME and improves accuracy up to 88\% over the SWE.
\end{abstract}




\vspace{5mm}
\vspace{2mm}

\vspace{5mm}



\section{Introduction}\label{section:Introduction}
The Shallow Water Equations (SWE) are widely applied to model free-surface flows in contexts such as ocean modelling \cite{Miller_2007} and snow avalanche simulation \cite{Sanz-Ramos_Blade_Oller_Furdada_2023}. These equations result from depth-averaging the Navier–Stokes equations, assuming a simplified vertical velocity profile represented by the depth-averaged velocity \cite{Vreugdenhil1994}. A significant limitation of this approach is the inability to model the detailed evolution of vertical structure, which can lead to modelling errors in scenarios with non-uniform vertical velocity profiles such as tsunami propagation \cite{Arcas2011} or dam-break flows \cite{Liang2010}. Furthermore, in the presence of bottom friction, vertical velocity profiles become non-uniform. As a result, applying the SWE introduces errors, especially in granular flows \cite{Kern2009,Sanvitale2016,Schaefer2013} like avalanches and sediment transport scenarios, see references within \cite{JosGarresDaz2021}. {\color{black} In those settings different empirical bed friction laws are used to parametrize shear stress \cite{Elgamal2021,Lajeunesse2010}, while the coefficients thereof are typically only calibrated by fitting a fully-developed logarithmic profile under idealized conditions, for example steady and uniform flow, even though the flow can locally exhibit unsteady and non-uniform behaviour}. Hence, the SWE are insufficient for modelling non-uniform vertical velocity profiles. As a possible remedy, the Shallow Water Moment Equations (SWME) have been derived to approximate vertical velocity profiles within a model for shallow flows \cite{Kowalski}.

Within the SWME framework, vertical velocity profiles are approximated by expanding them in Legendre polynomials, with the expansion coefficients referred to as moment variables.
In \cite{Kowalski}, the SWME are derived from the incompressible hydrostatic Navier-Stokes equations. The derivation employs the method of moments, where the horizontal velocity is expanded in vertical direction using a Legendre polynomial basis, and a Galerkin projection is then applied to derive the evolution equations for the moment coefficients. Subsequently, the moment model concept was also applied to non-hydrostatic cases \cite{Scholz2023}, non-Newtonian cases with sediment transport \cite{JosGarresDaz2021}, and two-dimensional radially symmetric free-surface flows \cite{Verbiest2025}. Additionally, as discussed in \cite{Bauerle2025}, the two-dimensional SWME are analysed, focussing on their mathematical structure. In \cite{Kowalski, Scholz2023, Verbiest2025}, it was shown that increasing the number of moments yields solutions that more closely approximate reference solutions computed via a direct discretisation of the Navier-Stokes equations, particularly when compared to the SWE. Thus, the SWME offer a clear advantage in enhancing the accuracy of free-surface flow models. 

In \cite{JulianKoellermeier2020}, it has been demonstrated that for specific cases, the eigenvalues of the transport matrix exhibit non-zero imaginary components. Consequently, the propagation speeds in the SWME system are no longer real and instabilities can occur. To address this loss of hyperbolicity, the SWME transport matrix is linearised around a linear velocity profile resulting in the Hyperbolic Shallow Water Moment Equations (HSWME) \cite{JulianKoellermeier2020}. Additionally, the Shallow Water Linearised Moment Equations (SWLME) provide an alternative hyperbolic model by omitting non-linear terms in the SWME-system \cite{Koellermeier2022}.

The SWME, HSWME, and SWLME incorporate more variables than the SWE, making these moment models computationally more expensive. This increase arises because the SWME, HSWME, and SWMLE each incorporate additional moment variables, rather than only fluid height and averaged velocity variables. To mitigate this increase in computational expense, we aim to develop a method to reduce the number of variables in the SWME. To achieve this reduction, the friction term must be considered, as it determines which model variables can be simplified through an appropriate closure.

We must therefore consider not only the transport component of the SWME but also the source term. The model in \cite{Kowalski} considers the simple case of a Newtonian friction term with viscosity and slip length parameters. In the SWME, this source-term accounts for both viscous dissipation inside the water and boundary effects at the fluid-solid interface.

In \cite{Huang2022}, the equilibria and stability of the hyperbolic HSWME and SWLME models are analysed, focussing on the influence of the slip length parameter in the source term. These models exhibit three distinct equilibrium manifolds, each characterised by specific source-term behaviour: (1) a water-at-rest state for finite viscosity and slip length, (2) a constant-velocity state when the slip length significantly exceeds the fluid height, and (3) a bottom-at-rest state when the slip length is much smaller than the fluid height. Through structural stability analysis and numerical simulations in \cite{Huang2022}, the water-at-rest and constant-velocity equilibria were shown to be stable, whereas the bottom-at-rest equilibrium could exhibit instability.

 The constant-velocity equilibrium state of the original SWME system from \cite{Kowalski} can be described without the additional polynomial coefficients that the SWME are solving for. When simulating perturbations from this equilibrium, the full SWME model incurs potentially high computational costs to resolve the small vertical variations in the form of the polynomial coefficients. On the other hand, a computationally faster SWE model fails to capture these variations because it relies solely on average velocity and fluid height. To balance runtime and accuracy, we aim to derive efficient approximate solutions near equilibria, so that solving a potentially large system for the polynomial coefficients is not necessary. These approximations emerge from systematically applying asymptotic analysis to the SWME for various scalings of the viscosity and the slip length.

For the asymptotic analysis a technique similar to Chapman Enskog expansion in kinetic theory from \cite{chapman1990mathematical} will be applied to the SWME with scalings of the slip length and viscosity friction parameters. For similar approaches applied to the Boltzmann Equations with various scalings of the non-dimensional Knudsen number, we refer to \cite{Bobylev2005, kremer2006methods, struchtrup2005macroscopic}. 
{\color{black}Furthermore, in the context of shallow water equations, the application of these asymptotic expansion techniques can be found in, for instance, \cite{Decoene2009, Gerbeau2001, LeMeur2015,Rodriguez2012}.}

Applying asymptotic expansions and deriving explicit closure relations for the higher-order terms reduces the numbers of variables in the SWME, yielding the Reduced Shallow Water Moment Equations (RSWME). This derivation process consists of substituting asymptotic expansions into the governing equations and performing an order-of-magnitude analysis to isolate the dominant terms. Deriving analytical expressions for a subset of the variables in terms of the remaining model variables and substituting them back into the SWME ultimately yields a closed system. We derive the RSWME up to an arbitrary order of moments. We show that the closed RSWME system is identical for large numbers of moments. Thus, numerical solutions for varying numbers of moments can be obtained by solving the same system of partial differential equations, significantly reducing computational cost to simulate free-surface flows near equilibrium with varying numbers of moments. 

Additionally, for an arbitrary number of moments the hyperbolicity of the RSWME is analysed, and a hyperbolic regularisation procedure is developed. Numerical tests show that the new RSWME lead to a higher accuracy compared to the standard SWE, and a higher computational efficiency than the full SWME for flows close to equilibrium.

The rest of this paper is structured as follows: Section 2 presents the SWME system and derives asymptotic approximations for varying scalings of the viscosity and slip length parameters. In Section 3, the Chapman–Enskog expansion method is applied to a viscous slip scenario, resulting in the RSWME through the dimensional reduction of the SWME system. Subsequently, a hyperbolicity analysis of the RSWME is performed. Section 4 then presents numerical tests for smooth wave scenarios. Finally, Section 5 concludes the paper, summarising the key findings.

\section{Shallow Water Moment Equations}
This section briefly presents the derivation of the SWE and SWME models based on \cite{Kowalski}. The SWE and SWME systems are derived from the incompressible Navier-Stokes Equations (NSE).

\subsection{Shallow water model}
The Shallow Water Equations (SWE) are a set of hyperbolic partial differential equations that describe the behaviour of free-surface flows where horizontal length-scales significantly exceed vertical ones. The SWE are derived by depth-integrating the NSE. 
The classical SWE for a water column of fluid height $h$ and depth-averaged velocity $u_m$ in the horizontal direction $x$ are
\begin{align}
\begin{aligned}
    \partial_th +\partial_x(h u_m) &=0,\\
    \partial_t(hu_m)+\partial_x\left(hu_m^2+\frac{gh^2}{2}\right)+gh\partial_x(h_b) &=-\frac{\nu}{\lambda}u_m,
\end{aligned} \label{eq:SWEconv}
\end{align}
where $\nu$ and $\lambda$ denote viscosity and slip length, respectively, while the bottom height is given by $h_b$. The right-hand side term of \eqref{eq:SWEconv} represents the bottom slip penalty.

The SWE from \eqref{eq:SWEconv} provide a computationally efficient two-variable system for flow modelling through depth-averaging. However, depth-averaging leads to errors in cases where significant horizontal velocity variations occur \cite{Kowalski}.

\subsection{Deriving SWME from the Navier–Stokes equations}

This subsection presents the derivation of the SWME, as introduced in \cite{Kowalski}, where vertical velocity profiles are approximated using a polynomial expansion to account for vertical flow structure. To obtain the SWME system, we transform the NSE through the following two steps: (1) mapping the vertical coordinate \( z \), and (2) applying a polynomial Legendre expansion to the horizontal component of the flow velocity.

\paragraph{Vertical coordinate mapping} As presented in \cite{Kowalski}, a coordinate mapping is employed to enable a more efficient representation of horizontal velocity variations along the vertical axis. This mapping normalises the vertical coordinate \( z \), between the bottom fluid height \( z = h_b(x,t) \) and the free-surface fluid height \( z = h_s(x,t) \) as 
\begin{align}
    \zeta(x,t,z) = \frac{z - h_b(x,t)}{h(x,t)}, \label{eq:mapping}
\end{align}
where $h(x,t)=h_b(x,t)-h_s(x,t)$ is the fluid height.
The mapped vertical coordinate $\zeta $ from \eqref{eq:mapping} is defined within the interval $[0,1]$. As a result, the $z$-coordinate, where $z\in [h_b(x,t),h_s(x,t)]$, is mapped to a normalised domain \( \zeta(x,t) \colon\text{ }  \mathbb{R}\times [0,T] \to [0,1] \).

\paragraph{Polynomial expansion} In the second step, \cite{Kowalski} expands the horizontal velocity variable \( u(x,t,\zeta) \) in the vertical direction $\zeta$ using a polynomial expansion centred around the depth-averaged velocity \( u_m(x,t) \). The horizontal component of the flow velocity \( u(x,t,\zeta) \colon\text{ } \mathbb{R} \times [0,T] \times [0,1] \to \mathbb{R} \) is expressed as
\begin{align}
    u(x,t,\zeta) = u_m(x,t) + \sum_{j=1}^N \alpha_j(x,t) \phi_j(\zeta), \label{eq:velocity_expansion}
\end{align}
where \( \alpha_j(x,t) \colon \text{ }\mathbb{R} \times [0,T] \to \mathbb{R} \) represents the expansion coefficients called moments, and \( \phi_j \colon \text{ } [0,1] \to \mathbb{R}\) denotes the scaled Legendre polynomial of degree \( j \) for \( j = 0, \ldots, N \). These basis functions are normalised such that \( \phi_j(0) = 1 \) for \( j = 1, \ldots, N \), {\color{black} and can be computed using the following formula:}
\begin{align}
    \phi_j(\zeta) = \frac{1}{j!} \frac{d^j}{d\zeta^j} \left( \zeta - \zeta^2 \right)^j, \qquad \text{for all } j=1,\ldots,N. \label{eq:scaled_legendre_def}
\end{align}
For example, the first two polynomials are
\begin{align}
    \phi_1(\zeta) = 1 - 2\zeta,\quad
    \phi_2(\zeta) = 1 - 6\zeta + 6\zeta^2.
\end{align}
\\
The basis functions satisfy the orthogonality condition 
\begin{align}
    \int_0^1 \phi_j(\zeta) \phi_k(\zeta) \, \mathrm{d}\zeta = \frac{\delta_{jk}}{2k+1}, \label{eq:orthogonality}
\end{align}
where \( \delta_{jk} \) is the Kronecker delta for all $j,k = 1,\ldots,N$.

The orthogonality property \eqref{eq:orthogonality} is essential to derive the SWME. This property leads to a separate time-derivative term for the $i$-th moment, thereby causing the nonlinear terms to couple the moment variables. As a result, the orthogonality property yields a sparse SWME where additional nonlinear terms do not appear in the time derivative.

In the expansion \eqref{eq:velocity_expansion}, the parameter \( N \in \mathbb{N} \) determines the order of the polynomial approximation. Increasing $N$ generally improves the representation of vertically complex flow behaviour \cite{Kowalski}, whereas setting \( N = 0 \) results in the classical depth-averaged SWE from \eqref{eq:SWEconv}.

After mapping the Navier-Stokes momentum equation in the $z$-direction and substituting the polynomial expansion formula from \eqref{eq:velocity_expansion}, the momentum equation is projected onto Legendre polynomials. Multiplying by $\phi_j$ and integrating over $\zeta \in [0,1]$ yields differential equations for $\alpha_j$ ($j=1,\dots,N$), while projecting onto $\phi_0=1$ gives the equation for $u_m$. The evolution equation for $h$ is derived by depth-integrating the vertically mapped continuity equation. The resulting $N+2$ equations for the variables $h, u_m,\alpha_1,\dots,\alpha_N$ are

\begin{subequations} \label{syst:IntroSWME}
\begin{align}
    \partial_t h + \partial_x(hu_m) &\,=\, 0, \label{eq:IntroSWME1} \\
    \partial_t(hu_m) + \partial_x \left(h \left(u_m^2 + \sum_{j=1}^N \frac{\alpha_j^2}{2j+1}\right) + \frac{g h^2}{2}\right) \label{eq:IntroSWME2} 
    &\,=\, -\frac{\nu}{\lambda} \left(u_m + \sum_{j=1}^N \alpha_j\right)- gh \partial_x(h_b), \\
    \partial_t(h \alpha_i) + \partial_x \left(h \left(2u_m \alpha_i + \sum_{j,k=1}^N A_{ijk} \alpha_j \alpha_k\right)\right) \label{eq:IntroSWME3}  &\,=\, u_m \partial_x(h \alpha_i) - \sum_{j,k=1}^N B_{ijk} \partial_x(h \alpha_j) \alpha_k \\
    &\qquad - (2i+1)\frac{\nu}{\lambda} \left(u_m + \sum_{j=1}^N \alpha_j + \frac{\lambda}{h} \sum_{j=1}^N C_{ij} \alpha_j\right), \nonumber
\end{align}
\end{subequations}
for \( i = 1, \ldots, N \). 

For simplicity, we assume uniform bathymetry, so $\partial_x(h_b)=0$. The constants $A_{ijk}$, $B_{ijk}$, and $C_{ij}$ for \( i, j, k = 1, \, \ldots, \, N \) denote integrals arising from scaled depth integration. These constant coefficients are given by
\begin{align}
    A_{ijk} &= (2i + 1) \int_0^1 \phi_i(\zeta) \phi_j (\zeta)\phi_k (\zeta)\, \mathrm{d}\zeta, \\
    B_{ijk} &= (2i + 1) \int_0^1 \partial_{\zeta} \phi_i (\zeta) \left( \int_0^\zeta \phi_j (\hat{\zeta})\, \mathrm{d}\hat{\zeta} \right) \phi_k(\zeta) \, \mathrm{d}\zeta, \\
    C_{ij} &= (2i + 1) \int_0^1 \partial_{\zeta} \phi_i (\zeta)\partial_{\zeta} \phi_j (\zeta)\, \mathrm{d}\zeta.\label{eq:Cconst}
\end{align}

\paragraph{Closed-form SWME} The SWME \eqref{eq:IntroSWME1}-\eqref{eq:IntroSWME3} can be expressed in matrix form as 
\begin{equation}
\partial_tU + \frac{\partial F(U)}{\partial U} \partial_x U = Q(U) \partial_xU + S(U), \label{eq:Systemnoncon}
\end{equation}
where \( U = (h, hu_m, h\alpha_1, \dots, h\alpha_N) \in \mathbb{R}^{N+2} \) is the vector of conservative variables. 
Here, the Jacobian of the conservative flux \(\frac{\partial F(U)}{\partial U} \in \mathbb{R}^{(N+2) \times (N+2)}\) is defined as 
\begin{align*}
\frac{\partial F(U)}{\partial U} =
{\begin{pmatrix}
0&1&0&\cdots &0\\
gh-u_m^2-\displaystyle{\sum_{i=1}^N\frac{\alpha_i}{2i+1}}&2u_m&\dfrac{2\alpha_1}{2\cdot1+1}&\cdots&\dfrac{2\alpha_N}{2N+1}\\
-2u_m\alpha_1-\displaystyle{\sum_{j,k=1}^NA_{1jk}\alpha_j\alpha_k}&2\alpha_1&2u_m\delta_{11}+2\displaystyle{\sum_{k=1}^NA_{11k}\alpha_k}&\cdots&2u_m\delta_{1N}+2\displaystyle{\sum_{k=1}^NA_{1Nk}\alpha_k}\\
\vdots&\vdots&\vdots&\ddots&\vdots\\
-2u_m\alpha_N-\displaystyle{\sum_{j,k=1}^NA_{Njk}\alpha_j\alpha_k}&2\alpha_N&2u_m\delta_{N1}+2\displaystyle{\sum_{k=1}^NA_{N1k}\alpha_k}&\cdots&2u_m\delta_{NN}+2\displaystyle{\sum_{k=1}^NA_{NNk}\alpha_k}
\end{pmatrix}}.
\end{align*}
The non-conservative term \(Q(U)\in \mathbb{R}^{(N+2) \times (N+2)}\) is given by
\begin{align*}
    Q(U) = \begin{pmatrix}
        0&0&0&\cdots&0\\
        0&0&0&\cdots&0\\
        0&0&u_m\delta_{11}+\displaystyle{\sum_{k=1}^NB_{11k}\alpha_k}&\cdots&u_m\delta_{1N}+\displaystyle{\sum_{k=1}^NB_{1Nk}\alpha_k}\\
        \vdots&\vdots&\vdots&\ddots&\vdots\\
        0&0&u_m\delta_{N1}+\displaystyle{\sum_{k=1}^NB_{N1k}\alpha_k}&\cdots&u_m\delta_{NN}+\displaystyle{\sum_{k=1}^NB_{NNk}\alpha_k}\\
    \end{pmatrix}.
\end{align*}
The friction term \(S(U)\in \mathbb{R}^{N+2}\) includes the Newtonian slip terms from \cite{Kowalski} and is determined by
\begin{align}
\begin{aligned}
    S_0(U) = 0, \quad S_1(U) = -\frac{\nu}{\lambda}\displaystyle{\sum_{j=1}^N\alpha_j},\quad
    S_i(U) = -(2i+1)\frac{\nu}{\lambda}\displaystyle{\sum_{j=1}^N\alpha_j}-(2i+1)\frac{\nu}{h}\displaystyle{\sum_{j=1}^NC_{ij}\alpha_j}, \quad \text{for } i=2,\dots,N+2.
    \end{aligned}\label{eq:SourceTerm}
\end{align}
The system matrix $A(U)\in  \mathbb{R}^{(N+2) \times (N+2)}$ is defined as
\begin{align}
    A(U)=\frac{\partial F(U)}{\partial U}-Q(U). \label{eq:SystemA}
\end{align}
Consequently, the full system \eqref{eq:Systemnoncon} is expressed as
\begin{align}
    \partial_tU + A(U) \partial_x(U) = S(U).\label{eq:SystemPDE}
\end{align}

\section{Asymptotic approximations of the SWME with viscous slip scaling}

In this section we present asymptotic approximations of the SWME using Chapman Enskog expansions. To this end, we introduce a small positive parameter, \( \varepsilon \ll 1 \), and apply a scaling to the friction parameters, \( \nu \) and \( \lambda \) in \eqref{eq:SourceTerm}. We apply asymptotic expansions to the variables in the system \eqref{eq:SystemPDE}, according to the expression
\begin{align}
    U_i=U_i^{(0)}+\varepsilon U_i^{(1)}+\varepsilon^2U_i^{(2)}+\mathcal{O}(\varepsilon^3) \qquad \text{for all } i =1,\ldots,N+2. \label{eq:generalexp}
\end{align}
The values of the different magnitude coefficients $U_i^{(j)}$ for $j=0,1,2$ are obtained by substituting \eqref{eq:generalexp} into \eqref{eq:IntroSWME1}--\eqref{eq:IntroSWME3} or equivalently \eqref{eq:SystemPDE} and systematically solving for varying orders of $\varepsilon$. We observe that as $\varepsilon \to 0$, the asymptotic expansion \eqref{eq:generalexp} reduces to
\begin{equation}
    U_i = U_i^{(0)} \qquad \text{for all }i = 1, \dots, N+2,
\end{equation}
which we refer to as the leading-order solution.
Subsequently, substituting asymptotic approximations yields the dimensionally reduced closure of the SWME.

In this study, a viscous slip scaling is chosen such that the viscosity is $\nu = \mathcal{O}(\varepsilon^{-1})$, and the slip length is $\lambda = \mathcal{O}(\varepsilon^{-1})$. We introduce the following scalings, given by
\begin{align}
\lambda = \frac{\lambda_0}{\varepsilon},\quad \nu = \frac{\nu_0}{\varepsilon}, \label{eq:blah}
\end{align}
where $\lambda_0$ and $\nu_0$ are assumed dimensionless.
The choice of a large slip length $\lambda$ is motivated by the findings in \cite{Huang2022}, where the limit $\lambda \to \infty$ corresponds to a stable equilibrium with a constant velocity profile in the vertical direction and vanishing moment variables. In the vicinity of this equilibrium, where $\lambda \gg 1$ but remains finite, we investigate approximations of the moment variables $\alpha_j$ for $i=1,\ldots,N$. Moreover, we prove that by selecting a large viscosity $\nu=\mathcal{\mathcal{O}}(\varepsilon^{-1})$, we can obtain a closure of the SWME model with dimensional reduction. In Section~\ref{sec:alternative:scalings}, alternative scalings of friction parameters are discussed. 

Applying the viscous slip scaling \eqref{eq:blah} to Equation \eqref{eq:IntroSWME3} and letting $\varepsilon\rightarrow 0$ yields
\begin{align}
    \sum_{j=1}^NC_{ij}\alpha_j=0,\quad \text{for } i =1,\ldots,N.\label{eq:leadingorder1}
\end{align}
Equation \eqref{eq:leadingorder1} can be expressed in matrix-vector form as
\begin{align}
    \boldsymbol{C}_{N}\boldsymbol{\alpha}_N = \boldsymbol{0}, \label{eq:leadingorder2}
\end{align}
where $\boldsymbol{C}_{N} \in \mathbb{R}^{N \times N}$ is a matrix with entries $(\boldsymbol{C}_N)_{ij} = C_{ij}$ and $\boldsymbol{\alpha}_N \in \mathbb{R}^N$ is a vector with entries $(\boldsymbol{\alpha}_N)_j = \alpha_j$, for $i, j = 1, \ldots, N$.
  According to \cite{Huang2022}, the matrix $\boldsymbol{C}_{N}$ is invertible, therefore the only possible leading order solution to \eqref{eq:leadingorder1} is $\alpha_j =0$ for $j=1,\ldots,N$. 

In \cite{Huang2022}, letting the slip variable $\lambda \rightarrow\infty$ results in a stable constant-velocity equilibrium in the HSWME and SWLME systems such that $\alpha_j =0$ for $ j=1,\dots,N$. This large slip equilibrium corresponds to the viscous slip leading order solution from \eqref{eq:leadingorder2}. In this study, our objective is to derive approximations for the viscous slip SWME system in the vicinity of $\alpha_j =0$ for $ j=1,\dots,N$. 

\textcolor{black}{
For alternative scalings of the friction parameters $\lambda$ and $\nu$, leading-order solutions exist. However, this work does not provide a detailed asymptotic analysis of these cases. Alternative scalings are presented in the following subsection.}

\subsection{Alternative Scalings} \label{sec:alternative:scalings}
\color{black}{
For different scalings of the friction parameters $\lambda$ and $\nu$, approximations near their respective equilibria can be derived. Table \ref{tab:VariousCases} summarises the feasible cases, demonstrating that such approximations can only be obtained when the viscosity $\nu$ is of order $\mathcal{O}(\varepsilon^{-1})$ or the slip length $\lambda$ is of order $\mathcal{O}(\varepsilon)$. In this table the fourth column indicates the order of the asymptotic solution, while the fifth column lists the variables retained in the dimensionally reduced system near equilibrium.

A possible strategy presented in Case II to reduce the number of variables is to scale the slip length by $\lambda= \lambda_0\varepsilon$. In the no-slip regime in Case II, the leading order solution, as $\varepsilon\rightarrow0$ in \eqref{eq:IntroSWME2} and \eqref{eq:IntroSWME3}, reduces to
\begin{align}
    u_m+\alpha_1+\cdots+\alpha_N=0. \label{eq:leadingordII}
\end{align}
This leading order approximation \eqref{eq:leadingordII} coincides with the no-slip boundary condition $u|_{\zeta=0}=0$. 
Then, an asymptotic approximation exists, but reduces the system only by a single variable through the approximation of $\alpha_N$ as a function of $h$, $u_m$, and $\alpha_1, \dots, \alpha_{N-1}$, resulting in approximations of order $\mathcal{O}(\varepsilon^0)$.
\textcolor{black}{
Note that leading-order approximations of $\mathcal{O}(\varepsilon^0)$, correspond to systems in which one boundary condition is strongly enforced. Specifically, enforcing the leading-order solutions from \eqref{eq:leadingordII} is equivalent to imposing the bottom boundary condition in the polynomial ansatz \eqref{eq:velocity_expansion} and letting $\lambda \to 0$. Imposing only the bottom boundary condition yields a moment model with $N+1$ variables instead of $N+2$. In \cite{Kowalski}, both the bottom and surface boundary conditions are strongly enforced (with $\lambda = 0$), resulting in a moment model with only $N$ variables. In what follows we investigate the extent to which dimensional reduction can be achieved within the SWME framework. We restrict our attention to cases other than the no-slip boundary condition and therefore do not consider asymptotic expansions of the no-slip case in this work.}

In Cases III and IV in Table \ref{tab:VariousCases}, the scalings $\nu=\frac{\lambda_0}{\varepsilon}$ and $\nu=\frac{\lambda_0}{\varepsilon}$, $\lambda = \lambda_0\varepsilon$ are applied, respectively. Letting $\varepsilon\rightarrow0$ in \eqref{syst:IntroSWME}, results in the leading order equations
\begin{align}
    u_m=\alpha_1=\cdots=\alpha_N=0,\qquad \partial_th =0.
\end{align}
Then, the application of asymptotic approximations ultimately produces a system reduced by one variable $h$.
Additionally, the approximated solutions for $u_m$ and $\alpha_j$ for $j=1,\dots,N$ are of order $\mathcal{O}(\varepsilon)$, indicating that these solutions are already close to a steady-state equilibrium where $h$ is constant and $u = 0$. 

\begin{table}[htbp]
    \centering
    \caption{Possible scalings for slip length $\lambda$ and viscosity $\nu$.}
    \label{tab:VariousCases}
    \begin{tabular}{ccccc}
        \toprule
        \textbf{Case} & $\boldsymbol{\lambda}$ & $\boldsymbol{\nu}$ & \textbf{Order of solution} & \textbf{Variables for $\mathcal{E}$} \\
        \midrule
         I & $\mathcal{O}(\varepsilon^{-1})$ & $\mathcal{O}(\varepsilon^{-1})$ & $\mathcal{O}(\varepsilon^2)$ & $h, u_m$ \\
        II  & $\mathcal{O}(\varepsilon)$  & $\mathcal{O}(\varepsilon^0)$ & $\mathcal{O}(\varepsilon^0)$ & $h, hu_m, h\alpha_1,\dots,h\alpha_{N-1}$ \\
        III & $\mathcal{O}(\varepsilon^0)$  & $\mathcal{O}(\varepsilon^{-1})$ & $\mathcal{O}(\varepsilon^3)$ & $h$ \\
        IV & $\mathcal{O}(\varepsilon)$  & $\mathcal{O}(\varepsilon^{-1})$ & $\mathcal{O}(\varepsilon^3)$ & $h$ \\
        \hline
    \end{tabular}
\end{table}
}

\subsection{Asymptotic Expansions}
\label{s:AsymptoticExpansions}
\textcolor{black}{We perform an asymptotic analysis in the high-viscous and high-slip regime (Case~I, Table~\ref{tab:VariousCases}). Under the scaling \eqref{eq:blah}, the equations governing $h$ \eqref{eq:IntroSWME1} and $hu_m$ \eqref{eq:IntroSWME2} are $\varepsilon$-independent.} \color{black}Consequently, taking the limit $\varepsilon \to 0$ in the SWME does not yield a leading-order solution for either $h$ or $hu_m$. Thus, this independence makes it impossible to derive asymptotic expansions for the variables $h$ and $u_m$. However, as demonstrated earlier, the equation for $\alpha_j$, for $j=1,\dots,N$, given by \eqref{eq:IntroSWME3}, does admit a leading-order solution under the scaling \eqref{eq:blah}. Subsequently, the variables $\alpha_i$ are asymptotically expanded as
\begin{equation}
    \alpha_i = \alpha_i^{(0)} + \varepsilon \alpha_i^{(1)} + \varepsilon^2 \alpha_i^{(2)} + \mathcal{O}(\varepsilon^3) \quad \text{for all } i=1,\dots,N.
    \label{eq:ChapmanEnskog}
\end{equation}
These expansions \eqref{eq:ChapmanEnskog} are then substituted into the governing system \eqref{eq:IntroSWME3}.

In the remainder of this section, an order-of-magnitude analysis is performed, in which the terms of the leading order, first order, and second order in $\varepsilon$ are systematically retained. This analysis enables a dimensional reduction of the system in the proximity of the viscous slip leading order solution, ensuring that the dynamics are captured at multiple scales of $\varepsilon$. The reduction process involves expressing the closure relations of $\alpha_j$ for $j=1,\ldots,N$ as functions of the remaining variables $h$ and $u_m$, incorporating combinations of terms at varying orders of $\varepsilon$. Finally, substituting these asymptotic expansions as closure relations back into \eqref{eq:IntroSWME3} closes the system, yielding the Reduced Shallow Water Moment Equation (RSWME).

We begin by deriving closure relations and closed systems for the SWME \eqref{eq:IntroSWME1}--\eqref{eq:IntroSWME3} under the viscous slip scaling \eqref{eq:blah}. The analysis proceeds systematically: in Subsection \ref{s:N1section}, we address the case \( N = 1 \); in Subsection \ref{s:N2section}, we extend the framework to \( N = 2 \); and in Subsection \ref{s:Ngensection}, we generalise the results to arbitrary \( N \). A key finding is that for \( N \geq 2 \), all closed RSWME systems are identical, {\color{black} although the polynomial expansion given by \eqref{eq:velocity_expansion} still changes with respect to $N$}. Additionally, we analyse the hyperbolicity of these systems for $ N = 1 $, $ N = 2 $, and arbitrary $ N $. Our analysis shows that none of these systems are globally hyperbolic. To resolve the loss of hyperbolicity, we introduce a hyperbolic regularisation.

\subsection{Asymptotic expansion for SWME with $N=1$}
\label{s:N1section}
{\color{black}In the thesis by D.~Wang \cite{Wang2023}, a preliminary asymptotic expansion for the coefficient $\alpha_1$ of the SWME \eqref{syst:IntroSWME} with $N=1$ is developed. We start by presenting the asymptotic expansion for $N=1$ to showcase} the main idea, before generalising to larger $N$. Let SWME1 denote the SWME model of order $N=1$. When the viscous slip scaling from  \eqref{eq:blah} is applied to SWME1, the governing equations become

\begin{align}
\begin{aligned}
    \label{eq:SWME_viscous_system}
    \partial_t \begin{pmatrix}
        h \\
        h u_m \\
        h \alpha_1
    \end{pmatrix}
    + \partial_x \begin{pmatrix}
        h u_m \\
        h u_m^2 + \frac{1}{3} h \alpha_1^2 + g \frac{h^2}{2} \\
        2 h u_m \alpha_1
    \end{pmatrix}
    &= \begin{pmatrix}
        0 \\
        0 \\
        u_m \partial_x (h \alpha_1)
    \end{pmatrix}
    - \frac{\nu_0}{\lambda_0} \begin{pmatrix}
        0 \\
        u_m + \alpha_1 \\
        3 \left( u_m + \alpha_1 + \frac{4 \lambda_0}{\varepsilon h} \alpha_1 \right)
    \end{pmatrix}.
    \end{aligned}
\end{align}
The last equation of \eqref{eq:SWME_viscous_system} depends on $\varepsilon$, which motivates expanding the variable $\alpha_1$ as
\begin{align}
\alpha_1 = \alpha_1^{(0)} + \varepsilon \alpha_1^{(1)} + \varepsilon^2 \alpha_1^{(2)} + \mathcal{O}(\varepsilon^3). \label{eq:alpha1_expansion}
\end{align}
By substituting \eqref{eq:alpha1_expansion} into the last equation of system \eqref{eq:SWME_viscous_system} and collecting terms of the same orders in \(\varepsilon\), we obtain explicit expressions for the asymptotic components \(\alpha_1^{(i)}\) (\(i = 0, 1, 2\)).

First, the \(\mathcal{O}(\varepsilon^{-1})\) terms in the last equation of \eqref{eq:SWME_viscous_system} are analysed to derive \(\alpha_1^{(0)}\). Next, \(\alpha_1^{(0)}\) is substituted back into the last equation of \eqref{eq:SWME_viscous_system}, and the \(\mathcal{O}(\varepsilon^{0})\) component is retained to determine the closure relation for \(\alpha_1^{(1)}\). Finally, the \(\mathcal{O}(\varepsilon^{1})\) component of \eqref{eq:SWME_viscous_system} is retained to derive the closure relation for \(\alpha_1^{(2)}\).

\paragraph{The $\mathcal{O}(\varepsilon^{-1})$-approximation of the SWME1} Let us substitute the asymptotic expansion \eqref{eq:alpha1_expansion} of \(\alpha_1\) into the last equation of \eqref{eq:SWME_viscous_system}. This substitution produces the \(\mathcal{O}(\varepsilon^{-1})\)-order component of \eqref{eq:SWME_viscous_system}, which is given by
\begin{equation}
    \label{eq:leading_order_alpha1}
    \frac{4 \lambda_0}{h} \alpha_1^{(0)} = 0
    \quad \Longrightarrow \quad
    \alpha_1^{(0)} = 0.
\end{equation}

\paragraph{The $\mathcal{O}(\varepsilon^{0})$-approximation of the SWME1} By retaining the \(\mathcal{O}(\varepsilon^{0})\)-terms of the last evolution equation in system \eqref{eq:SWME_viscous_system} and substituting \(\alpha_1^{(0)} = 0\), the closure relation for \(\alpha_1^{(1)}\) is obtained as 
\begin{align}
    \label{eq:first_order_alpha1}
    u_m + \frac{4 \lambda_0}{h} \alpha_1^{(1)} = 0
    \quad \Longrightarrow \quad
    \alpha_1^{(1)} = -\frac{u_m h}{4 \lambda_0}.
\end{align}

\paragraph{The $\mathcal{O}(\varepsilon^{1})$-approximation of the SWME1} {\color{black} Retaining the $\mathcal{O}(\varepsilon)$-order terms of \eqref{eq:SWME_viscous_system} yields a system that depends explicitly on the asymptotic components $\alpha_1^{(1)}$ and $\alpha_2^{(2)}$. 
Following this, substituting the closure relation for $\alpha_1^{(1)}$ from \eqref{eq:first_order_alpha1} into the $\mathcal{O}(\varepsilon)$ terms of \eqref{eq:SWME_viscous_system} we have}
\begin{equation}
    \label{eq:momentum_correction22}
    \partial_t (h^2 u_m) + 2 \partial_x (h^2 u_m^2) = u_m \partial_x (h^2 u_m) - 3 \frac{\nu_0}{\lambda_0} \left( u_m h - \frac{16 \lambda_0^2}{h} \alpha_1^{(2)} \right).
\end{equation}
{\color{black}
Furthermore, \ref{A:Prelim} presents the detailed derivation of the closure relation for $\alpha_1^{(2)}$ in the context of \eqref{eq:momentum_correction22}. The closure relation for coefficient $\alpha_1^{(2)}$ from \eqref{eq:expansion_alpha1}, is given by}

\begin{equation}
    \label{eq:alpha1_final2_t}
    \alpha_1^{(2)} = \frac{1}{48 \nu_0 \lambda_0} \left( -\frac{g}{4} \partial_x (h^4) + \frac{2 \nu_0}{\lambda_0} u_m h^2 \right).
\end{equation}

\paragraph{The closed reduced SWME1 model} Substituting the closure relations \eqref{eq:leading_order_alpha1}, \eqref{eq:first_order_alpha1}, and \textcolor{black}{\eqref{eq:alpha1_final2_t}}  into \eqref{eq:alpha1_expansion}  the closure relation for \(\alpha_1\) is derived as
\begin{align}
    \alpha_1 = \varepsilon \left( -\frac{h u_m}{4 \lambda_0} \right)
    + \varepsilon^2 \left[ \frac{1}{48 \nu_0 \lambda_0} \left( -\frac{g}{4} \partial_x (h^4) + 2 \frac{\nu_0 u_m h^2}{\lambda_0} \right) \right]
    + \mathcal{O}(\varepsilon^3). \label{eq:final1_alpha1}
\end{align}
To obtain a closed system for \(h\) and \(hu_m\), we substitute \(\alpha_1\) from \eqref{eq:final1_alpha1} into Equation \eqref{eq:SWME_viscous_system}. By neglecting terms of order \(\mathcal{O}(\varepsilon^3)\) and higher, the closed Reduced Shallow Water Moment Equations for \(N=1\) (RSWME1) are defined as \cite{Wang2023}
\begin{align}
\begin{aligned}
    \partial_th +\partial_x(hu_m) &= 0,\\
    \partial_t(hu_m) +\partial_x\left(hu_m^2T_1^1(h,u_m)+\frac{gh^2}{2}T_2^1(h,u_m)\right)&=-\frac{\nu_0 u_m}{\lambda_0}T_3^1(h,u_m),\label{eq: ASWME1}
\end{aligned}
\end{align}
where
\begin{align}
\begin{aligned}
T_1^1(h,u_m)=1+\frac{\varepsilon^2h^2}{48\lambda_0^2} ,\quad
T_2^1(h,u_m) = 1-\frac{\varepsilon^2  h^2}{96\lambda_0^2},\quad
    T_3^1(h,u_m) =1-\frac{\varepsilon h}{4\lambda_0}+\frac{\varepsilon^2 h^2}{24\lambda_0^2}.
\end{aligned}
\end{align}

\begin{remark} \label{rem:Remark1}
    The Boussinesq coefficient, which captures variations from a uniform velocity profile, is given by \cite{Liggett}
   \begin{align}
       \beta =\frac{1}{h u_m^2} \int_{h_b}^{h_s}u(x,z,t)^2\,\mathrm{d}z. \label{eq:boussisqcoeff}
   \end{align}
   In the context of the SWME1 framework in \cite{koellermeier2025energy}, this coefficient takes the form
   \begin{align}
       \beta =1+\frac{1}{3}\frac{\alpha_1^2}{u_m^2}.
   \end{align}
   Substituting the closure relation \eqref{eq:final1_alpha1} for $\alpha_1$ and ignoring the terms $\mathcal{O}(\varepsilon^3)$ and $\mathcal{O}(\varepsilon^4)$ yields 
   \begin{align}
       \beta \approx 1+\frac{1}{48}\frac{\varepsilon^2h^2}{\lambda_0^2}=T_1^1(h,u_m).
   \end{align}
   This finding means that the advection term of the RSWME1 is consistent with the approximation of a vertical velocity profile using the Boussinesq coefficient. This fact highlights the relevance of our model in view of similar models in the literature, see \cite{Liggett}.
\end{remark}

We derived a closed system \eqref{eq: ASWME1} for $h$ and $u_m$ that takes into account small deviations around the viscous slip leading order solution with a constant velocity profile using an explicit closure for $\alpha_1$ of $\mathcal{O}(\varepsilon^2)$ that was computed using the order-of-magnitude approach for terms up to $\mathcal{O}(\varepsilon^2)$. This means that the RSWME1 system is in theory able to approximate the SWME1 with a large viscosity and a large slip length while only numerically solving a system for two variables instead of three. 

The RSWME1 from \eqref{eq: ASWME1} depend on the same variables as the SWE \eqref{eq:SWEconv}. Furthermore, as \(\varepsilon \to 0\), the RSWME1 are consistent with the SWE. In contrast to the {SWE}, the friction term in the {RSWME1} incorporates additional contributions proportional to \(\varepsilon u_m h\) and \(\varepsilon^2 u_m h^2\), which represent variations around the mean velocity \(u_m\). Furthermore, the {RSWME1} features additional \(\mathcal{O}(\varepsilon^2)\)-linear combinations of \(h^3 u_m^2\) and \(h^2\) in the transport matrix on its left-hand side, which are not present in the {SWE}.

\color{black}
\paragraph{Hyperbolicity of the RSWME1} The RSWME1 system \eqref{eq: ASWME1} can be expressed in non-conservative form as
\begin{align}
    \partial_t U + A^{1}(U) \partial_x U = S^1(U), \label{eq:analyticalRSWME1}
\end{align}
where in this case \( U = (h, hu_m) \), and the system matrix \( A^1(U) \in \mathbb{R}^{2\times2}  \) is given by
\begin{align}
    A^1(U) = \begin{pmatrix}
        0 & 1 \\
        -u_m^2 \left(1 - \frac{\varepsilon^2h^2}{48\lambda_0^2}\right) + gh  - \frac{g\varepsilon^2 h^3}{48\lambda_0^2} & 2u_m \left(1 + \frac{\varepsilon^2 h^2}{48\lambda_0^2}\right)
    \end{pmatrix}. \label{eq:SystemMat1}
\end{align}
Additionally, the source term $S^1(U)\in \mathbb{R}^{2}$ is defined as
\begin{align}
S^1(U) = \begin{pmatrix}
    0\\
    -\frac{\nu_0u_m}{\lambda_0}\left(1-\frac{\varepsilon h}{4\lambda_0}+\frac{\varepsilon^2h^2}{24\lambda_0^2}\right)
\end{pmatrix}.
\end{align}
The RSWME1 (in the form \eqref{eq:analyticalRSWME1}) enable a deeper analysis of the system. The eigenvalues of the matrix $A^1(U)$ from \eqref{eq:SystemMat1} correspond to the propagation speeds of the waves, a critical property of the system. If these eigenvalues are complex, real propagation speeds no longer exist, potentially leading to instabilities. In this subsection, we investigate whether hyperbolicity is preserved in \eqref{eq:analyticalRSWME1}.

To analyse hyperbolicity, we determine whether the eigenvalues $\lambda_{1,2}$ of \( A^1(U) \) from \eqref{eq:SystemMat1} are real and distinct. These eigenvalues are given by
 \begin{equation}
    \lambda_{1,2} = u_m \left(1 + \frac{h^2 \varepsilon^2}{48\lambda_0^2}\right) \pm \sqrt{gh} \sqrt{\frac{\varepsilon^4 h^3 u_m^2}{48^2\lambda_0^4g} - \frac{\varepsilon^2 h^2 }{48\lambda_0^2} + \frac{\varepsilon^2 h u_m^2}{16\lambda_0^2g} + 1}.\label{eq:SqrtIa}
\end{equation}
For hyperbolicity, a positive discriminant is required, which is
\begin{align}
    \left(\frac{\varepsilon^4 h^3 }{48^2\lambda_0^4g}+\frac{\varepsilon^2 h}{16\lambda_0^2g}\right)u_m^2 - \frac{\varepsilon^2 h^2 }{48\lambda_0^2}   + 1 \geq 0.\label{eq:Sqrt1}   
\end{align}
Given the difficulty of determining \( h \) and $u_m$ for which inequality \eqref{eq:Sqrt1} holds, we first analyse the case where \( u_m = 0 \). For hyperbolicity in this specific case, the following condition on $h$ must hold
\begin{equation}
    0 < h < \frac{4\sqrt{3}\lambda_0}{\varepsilon}. \label{eq:condition1}
\end{equation}
Because the terms in \eqref{eq:Sqrt1} depending on \( u_m^2 \)  increase monotonically with \( u_m \), hyperbolicity at \( u_m = 0 \) means that hyperbolicity holds for all \( u_m>0 \). Condition \eqref{eq:condition1} imposes that \(\frac{h}{\lambda} = \frac{h\varepsilon}{\lambda_0} = \mathcal{O}(\varepsilon)\) must remain below a threshold, in line with the scaling (\(\varepsilon \ll 1\)). {\color{black}In the regime where \( h \) is sufficiently large such that condition \eqref{eq:condition1} breaks down, \( \frac{\lambda}{h} \) no longer scales as \( \mathcal{O}(\varepsilon^{-1}) \), causing the friction term in \eqref{eq:IntroSWME3} to be no longer the dominant contribution to the underlying SWME1 model.}  Consequently, hyperbolicity breaks down when the RSWME1 no longer capture the viscous slip regime. 

\begin{theorem}[Hyperbolic regularisation of RSWME1]
\label{thm:hyperbolicity_reg1}
The RSWME1 system \eqref{eq: ASWME1} admits a globally hyperbolic regularisation by adding the following term to the left-hand side of its last equation:
\begin{align}
\partial_x\left(\frac{4 \varepsilon^4 h^6g}{6\cdot 192^2\lambda_0^4}\right).\label{eq:regularisation1}
\end{align}

\end{theorem}

\begin{proof}

To establish global hyperbolicity, we augment the last equation of system \eqref{eq: ASWME1} with the regularisation term \eqref{eq:regularisation1}, such that the modified system matrix \eqref{eq:SystemMat1} is given by
\begin{align}
    A^{1,H}= \begin{pmatrix}
        0 & 1 \\
        -u_m^2 \left(1 - \frac{\varepsilon^2h^2}{48\lambda_0^2}\right) + gh - \frac{\varepsilon^2 h^3g}{48\lambda_0^2}+\frac{4\varepsilon^4h^5g}{192^2 \lambda_0^4 } & 2u_m \left(1 + \frac{\varepsilon^2 h^2}{48\lambda_0^2}\right)
    \end{pmatrix}.\label{eq:systemHRSWME1}
\end{align}
The eigenvalues of the system matrix \eqref{eq:systemHRSWME1} of the Hyperbolic RSWME1 (HRSWME1) are
\begin{align}
\begin{aligned}
\lambda_{1,2} = u_m \left(1 + \frac{h^2 \varepsilon^2}{48\lambda_0^2}\right)\pm \sqrt{gh} \sqrt{\frac{\varepsilon^4 h^3 u_m^2}{48^2\lambda_0^4g} - \frac{\varepsilon^2 h^2}{48\lambda_0^2} + \frac{\varepsilon^2 h u_m^2}{16\lambda_0^2g} + 1 + \frac{4\varepsilon^4  h^4}{192^2\lambda_0^4}}. \end{aligned}\label{eq:hyperolicsqrt}
\end{align}
For the special case where \(u_m = 0\), the eigenvalues are
\begin{align}
    \lambda_{1,2} = \pm \sqrt{gh}\left(1-\frac{2h^2}{192\lambda_0^2}\varepsilon^2\right).\label{eq:hyperolicsqrtI}
\end{align}
The term under the square root in \eqref{eq:hyperolicsqrt} is strictly positive for all $h > 0$ when $u_m = 0$, and hence this value remains positive for all $u_m$ as the terms containing $u_m$ are greater or equal than zero for $h>0$. Thus, for all $h > 0$ and $u_m$, the eigenvalues of \eqref{eq:systemHRSWME1} are real and distinct, ensuring the global hyperbolicity of HRSWME1.

\end{proof}
It is important to note that the inclusion of the $\mathcal{O}(\varepsilon^4)$-term \eqref{eq:regularisation1} does not reduce the accuracy of the HRSWME1, given that the RSWME1 are of order $\mathcal{O}(\varepsilon^2)$.

\subsection{Asymptotic expansion for SWME with $N=2$}
\label{s:N2section}
In this subsection, the $\mathcal{O}(\varepsilon^2)$-order SWME for $N=2$ (SWME2) are simplified using asymptotic approximations, in preparation for deriving a reduced SWME system for arbitrary $N$.
 This reduction decreases the system from four to two equations, analogous to the previously derived $N=1$ case.
The viscous slip system SWME2 based on \eqref{eq:IntroSWME1}--\eqref{eq:IntroSWME3} are given by
\begin{align}
\label{eq: Wanggen}
\begin{aligned}
    \partial_t \begin{pmatrix}
        h  \\ h u_m \\
        h\alpha_1\\
        h \alpha_2
    \end{pmatrix}+
    \partial_x \begin{pmatrix}
        h u_m \\ h u_m^2 +\frac{1}{3}h \alpha_1^2+\frac{1}{5}h \alpha_2^2+ g\frac{h^2}{2}\\
        2h u_m \alpha_1+ \frac{4}{5}\alpha_1\alpha_2\\
        2 hu_m \alpha_2 +\frac{2}{3}h \alpha_1^2+\frac{2}{7}h\alpha_2^2
        \end{pmatrix}
        &= \begin{pmatrix}
            0\\
            0\\
            (u_m-\frac{1}{5}\alpha_2)\partial_x(h\alpha_1)+\frac{1}{5}\alpha_1\partial_x(h\alpha_2)\\
            \alpha_1\partial(h\alpha_1)+(u_m+\frac{\alpha_2}{7})\partial_x(h\alpha_2)
        \end{pmatrix}\\
        &\quad-\frac{\nu_0}{\lambda_0}\begin{pmatrix}
            0 \\ u_m +\alpha_1+\alpha_2\\
            3(u_m +\alpha_1+\alpha_2+\frac{4 \lambda_0}{\varepsilon h}\alpha_1)\\
            5(u_m +\alpha_1+\alpha_2+\frac{12\lambda_0}{\varepsilon h}\alpha_2)
        \end{pmatrix}. 
        \end{aligned}
\end{align}
Similar to the RSWME1 case in Subsection \ref{s:N1section}, the asymptotic expansions for $\alpha_1$ and $\alpha_2$ are expressed as
\begin{align}
\alpha_1 &= \alpha_1^{(0)} + \varepsilon \alpha_1^{(1)} + \varepsilon^2 \alpha_1^{(2)} + \mathcal{O}(\varepsilon^3), \label{eq:expansion_alpha1} \\
\alpha_2 &= \alpha_2^{(0)} + \varepsilon \alpha_2^{(1)} + \varepsilon^2 \alpha_2^{(2)} + \mathcal{O}(\varepsilon^3). \label{eq:expansion_alpha2}
\end{align}
By substituting \eqref{eq:expansion_alpha1} and \eqref{eq:expansion_alpha2} into the final two equations of system \eqref{eq: Wanggen}, the coefficients $\alpha_1^{(i)}$ and $\alpha_2^{(i)}$ for $i = 0, 1, 2$ are systematically derived order by order. 

\paragraph{The $\mathcal{O}(\varepsilon^{-1})$ approximation of the SWME2}  Substitute the asymptotic expansions \eqref{eq:expansion_alpha1} and \eqref{eq:expansion_alpha2} into the final two equations of  \eqref{eq: Wanggen} and retain only the $\mathcal{O}(\varepsilon^{-1})$-order terms. Solving for $\alpha_1^{(0)}$ and $\alpha_2^{(0)}$ yields
\begin{equation}
\begin{pmatrix}
\alpha_1^{(0)} \\[0.7ex]
\alpha_2^{(0)}
\end{pmatrix}
= \mathbf{0}.
\label{eq:leading_order_solution2}
\end{equation}

\paragraph{The $\mathcal{O}(\varepsilon^0)$ approximation of the SWME2}
Retaining only the $\mathcal{O}(\varepsilon^{0})$-order terms in \eqref{eq: Wanggen} and incorporating \eqref{eq:leading_order_solution2}, we solve for $\alpha_1^{(1)}$ and $\alpha_2^{(1)}$ to obtain
\begin{align}
  \begin{pmatrix}
        \alpha_1^{(1)}\\[0.7ex]
        \alpha_2^{(1)}
    \end{pmatrix}=
    \begin{pmatrix}
        \frac{-hu_m}{4\lambda_0}\\[0.7ex]
        \frac{-hu_m}{12\lambda_0}
    \end{pmatrix}. \label{eq:viscous_contribution21}
\end{align}

\paragraph{The $\mathcal{O}(\varepsilon^1)$ approximation of the SWME2}
The $\mathcal{O}(\varepsilon)$-order contribution of \eqref{eq: Wanggen} yields closure relations
\begin{align}
   \begin{pmatrix}
\alpha_1^{(2)} \\[0.7ex]
\alpha_2^{(2)}
\end{pmatrix}
=
\frac{1}{\nu_0 \lambda_0}
\begin{pmatrix}
\frac{1}{48}\left(-\frac{g}{4} \partial_x(h^4) + 3 \frac{\nu_0}{\lambda_0} u_m h^2\right) \\[0.7ex]
\frac{1}{720}\left(-\frac{g}{4} \partial_x(h^4) + 19 \frac{\nu_0}{\lambda_0} u_m h^2\right)
\end{pmatrix}. \label{eq:viscous_contribution22}
\end{align}
The closure relations from \eqref{eq:viscous_contribution22} follow the derivation steps of $\alpha_1^{(2)}$ in Subsection \ref{s:N1section}.

\paragraph{The closed reduced SWME2 model}

Substituting the closure relations from  \eqref{eq:leading_order_solution2}, \eqref{eq:viscous_contribution21}, and \eqref{eq:viscous_contribution22} into the asymptotic expansions \eqref{eq:expansion_alpha1}--\eqref{eq:expansion_alpha2} leads to the closure relations for the coefficients \(\alpha_1\) and $\alpha_2$:
\begin{align}
    \alpha_1 &= \varepsilon \left( -\frac{h u_m}{4 \lambda_0} \right)
    + \varepsilon^2 \left( \frac{1}{48 \nu_0 \lambda_0} \left( -\frac{g}{4} \partial_x (h^4) + 3 \frac{\nu_0 }{\lambda_0} u_m h^2\right) \right)
    + \mathcal{O}(\varepsilon^3), \label{eq:final2_alpha1} \\
    \alpha_2 &= \varepsilon \left( -\frac{h u_m}{12 \lambda_0} \right)
    + \varepsilon^2 \left( \frac{1}{720 \nu_0 \lambda_0} \left( -\frac{g}{4} \partial_x (h^4) + 19 \frac{\nu_0 }{\lambda_0} u_m h^2\right) \right)
    + \mathcal{O}(\varepsilon^3). \label{eq:final2_alpha2}
\end{align}
To obtain the closed system for \(h\) and \(hu_m\), the expressions for $\alpha_1$ and \(\alpha_2\) from \eqref{eq:final2_alpha1} and \eqref{eq:final2_alpha2} are substituted into the second and third equations of \eqref{eq: Wanggen}. As we neglect higher order terms of order \(\mathcal{O}(\varepsilon^3)\) and higher, the closed Reduced SWME2 (RSWME2) are given by
\begin{align}
\begin{aligned}
    &\partial_th +\partial_x(hu_m) = 0,\\
    &\partial_t(hu_m) +\partial_x\left(hu_m^2T_1^2(h,u_m)+\frac{gh^2}{2}T_2^2(h,u_m)\right)=-\frac{\nu_0 u_m}{\lambda_0}T_3^2(h,u_m),\label{eq: ASWME2}
\end{aligned}
\end{align}
where
\begin{align*}T_1^2(h,u_m)=1+\frac{\varepsilon^2h^2}{45\lambda_0^2},\quad T_2^2(h,u_m) = 1-\frac{\varepsilon^2  h^2}{90\lambda_0^2},\quad
T_3^2(h,u_m) =1-\frac{\varepsilon h}{3\lambda_0}+\frac{\varepsilon^2 h^2}{45\lambda_0^2}.\end{align*} Observe that the RSWME2 in \eqref{eq: ASWME2} share the same structure as the RSWME1 from \eqref{eq: ASWME1}. The only difference arises from the constants that scale the $\mathcal{O}(\varepsilon)$ and $\mathcal{O}(\varepsilon^2)$ terms.

Observe that one can derive the Boussinesq coefficient for the RSWME2, as in Remark \ref{rem:Remark1}. Repeating this procedure will show that the RSWME2 system also has an advection term consistent with the Boussinesq-coefficient-based approximation of the vertical velocity profile. 

\paragraph{Hyperbolicity of the RSWME2}
The system in \eqref{eq: ASWME2} can be expressed in non-conservative form as
\begin{align}
    \partial_t U + A^2(U) \partial_x U = S^2(U),
\end{align}
where \( U = (h, hu_m) \). The system matrix \( A^2(U)\in \mathbb{R}^{2\times2}\) is given by
\begin{align}
    A^2(U) = \begin{pmatrix}
        0 & 1 \\
        -u_m^2 \left(1 - \frac{\varepsilon^2h^2}{45\lambda_0^2}\right) + gh  - \frac{g\varepsilon^2 h^3}{45\lambda_0^2} & 2u_m \left(1 + \frac{\varepsilon^2 h^2}{45\lambda_0^2}\right)
    \end{pmatrix}. \label{eq:SystemMat2}
\end{align}
Additionally, the source term is given by $S^2(U)\in\mathbb{R}^2$ as 
\begin{align}
    S^2(U) = \begin{pmatrix}
        0\\
        -\frac{\nu_0 u_m}{\lambda_0}\left(1-\frac{\varepsilon h}{3\lambda_0}+\frac{\varepsilon^2 h^2}{45\lambda_0^2}\right)
    \end{pmatrix}.
\end{align}
To analyse hyperbolicity, we determine whether the eigenvalues $\lambda_{1,2}$ of \( A^2(U) \) from \eqref{eq:SystemMat2} are real and distinct. These eigenvalues are given by
 \begin{equation}
    \lambda_{1,2} = u_m \left(1 + \frac{h^2 \varepsilon^2}{45\lambda_0^2}\right) \pm \sqrt{gh} \sqrt{\frac{\varepsilon^4 h^3 u_m^2}{2025\lambda_0^4g} - \frac{\varepsilon^2 h^2 }{45\lambda_0^2} + \frac{\varepsilon^2 h u_m^2}{15\lambda_0^2g} + 1}.
\end{equation}
For hyperbolicity, the condition that must be satisfied is
\begin{align}
    \left(\frac{\varepsilon^4 h^3 }{2025\lambda_0^4g}+\frac{\varepsilon^2 h}{15\lambda_0^2g}\right)u_m^2 - \frac{\varepsilon^2 h^2 }{45\lambda_0^2} +  1 \geq 0.\label{eq:Sqrt2}   
\end{align}
Similar to analysing the hyperbolicity of the RSWME1 system, we begin by analysing the hyperbolicity for $u_m=0$. For this specific case, hyperbolicity requires that the following condition on $h$ is satisfied
\begin{equation}
    0 < h < \frac{3\sqrt{5}\lambda_0}{\varepsilon}. \label{eq:condition2}
\end{equation}
The terms depending on $u_m$ on the left-hand side of Equation \eqref{eq:Sqrt2} are always positive for $h>0$ and any non-zero value of $u_m$. Satisfying condition \eqref{eq:condition2} in the case where  \( u_m = 0 \), ensures that hyperbolicity holds for all \( u_m \geq 0 \). 
Furthermore, condition \eqref{eq:condition2} imposes that \(\frac{h}{\lambda} = \frac{h\varepsilon}{\lambda_0} = \mathcal{O}(\varepsilon)\) must remain below a threshold, in line with the scaling (\(\varepsilon \ll 1\)). {\color{black}In the regime where \( h \) is sufficiently large such that condition \eqref{eq:condition2} breaks down, \( \lambda/h \) no longer scales as \( \mathcal{O}(\varepsilon^{-1}) \), causing the friction term in \eqref{eq:IntroSWME3} to be no longer the dominant contribution to the underlying SWME2 model.}  Consequently, hyperbolicity breaks down when the RSWME2 no longer capture the viscous slip regime. 

{\color{black}Note that, the hyperbolicity constraint \eqref{eq:condition2} for the RSWME2, is stricter than the constraint \eqref{eq:condition1} for the RSME1. This observation aligns with the underlying behavior of the SWME, where the hyperbolicity condition for \( N = 2 \) is more restrictive than for \( N = 1 \) \cite{Kowalski}.}

\color{black}

\begin{theorem}[Hyperbolic regularisation of RSWME2]
\label{thm:hyperbolicity_reg2}
The RSWME2 system, as defined in Equation \eqref{eq: ASWME2}, admits a globally hyperbolic regularisation through the addition of the following term to the left-hand side of its last equation:
\begin{align}
\partial_x\left(\frac{4 \varepsilon^4 h^6g}{6\cdot 180^2\lambda_0^4}\right).\label{eq:regularisation2}
\end{align}
\end{theorem}

\begin{proof}
   This result follows analogously to the proof of Theorem \ref{thm:hyperbolicity_reg1}.
\end{proof}

\subsection{Asymptotic expansion for SWME with arbitrary $N$}
\label{s:Ngensection}
This section presents the derivation of solutions for an arbitrary number of moments, denoted by \(N\). The order analysis for arbitrary $N$ employs asymptotic expansions of the form
\begin{equation}
\alpha_{N,i} = \alpha_{N,i}^{(0)} + \varepsilon \alpha_{N,i}^{(1)} + \varepsilon^2 \alpha_{N,i}^{(2)} + \mathcal{O}(\varepsilon^3) \qquad \text{for }i = 1, \dots, N, \label{eq:alphaNsol}
\end{equation}
and will be substituted in \eqref{eq:IntroSWME1}--\eqref{eq:IntroSWME3}. Note that we use the notation $\alpha_{N,i}$ to denote the moment variable $\alpha_i$ related to the SWME of order $N$. In the remainder of this subsection, we derive the closure relations for $\alpha_{N,i}$, which may also vary depending on $N$.

\paragraph{The $\mathcal{O}(\varepsilon^{-1})$ approximation of the SWME}
Substitution of the asymptotic expansions \eqref{eq:alphaNsol} for \(i = 1, \ldots, N\) into Equation \eqref{eq:IntroSWME3} yields the \(\mathcal{O}(\varepsilon^{-1})\)-order term
\begin{align}
    \frac{\nu_0}{h}\left(\sum_{j=1}^{N}C_{ij}\alpha_{N,j}^{(0)}\right)=0.\label{eq:zerosys}
\end{align}
The system \eqref{eq:zerosys} in matrix-vector form is
\begin{align}
    \boldsymbol{C}_{N} \boldsymbol{\alpha}_N^{(0)}=\mathbf{0},\label{eq:MatrixC1}
\end{align}  
where $\boldsymbol{C}_{N} \in \mathbb{R}^{N\times N}$ is as defined in \eqref{eq:leadingorder2} and  \(\boldsymbol{\alpha}_N^{(0)}\) is the vector whose \(j\)-th component corresponds to \(\alpha_{N,j}^{(0)}\).
As shown in \cite{Huang2022}, the matrix \(\boldsymbol{C}_{N}\) is invertible, and as a result the unique solution to \eqref{eq:MatrixC1} is
\begin{equation}
\alpha_{N,i}^{(0)} = 0 \qquad \text{for all }i = 1, \dots, N.
\label{eq:trivial_solution}
\end{equation}

\paragraph{The $\mathcal{O}(\varepsilon^{0})$ approximation of the SWME}
 Substituting \eqref{eq:trivial_solution} into the $\mathcal{O}(\varepsilon^0)$-order system \eqref{eq:IntroSWME3} yields $N$ closure relations for the coefficient variable $\alpha_{N,j}^{(1)}$:
\begin{align*}
    \sum_{j=1}^N C_{ij} \alpha_{N,j}^{(1)}= -\frac{u_mh}{\lambda_0},
\end{align*}
for all $i=1,\ldots,N$.
In matrix-vector form, the solution for $\boldsymbol{\alpha}_N^{(1)} \in \mathbb{R}^N$ can be expressed as
\begin{align}
    \boldsymbol{C}_{N} \boldsymbol{\alpha}_N^{(1)}= -\frac{u_mh}{\lambda_0} \mathbf{1}, \label{eq:MatrixC2}
\end{align}
where \(\boldsymbol{\alpha}_N^{(1)}\) is the vector of first-order coefficients with the following entries \((\boldsymbol{\alpha}_N^{(1)})_i = \alpha_{N,i}^{(1)}\) for $i=1,\dots,N$, and $\mathbf{1}\in \mathbb{R}^N$ is the vector of ones.
Rewriting Equation \eqref{eq:MatrixC2} yields closure relations for the coefficients \(\alpha_{N,i}^{(1)}\) for \(i = 1, \dots, N\):
\begin{align}
\alpha_{N,i}^{(1)} = -\frac{u_m h}{\lambda_0} \sum_{j=1}^N (\boldsymbol{C}_{N}^{-1})_{ij}.\label{eq:MatrixC22}
\end{align}
Thus, the closure relation \(\alpha_{N,i}^{(1)}\), for $i=1,\dots,N$ can be expressed as
\begin{align}
\alpha_{N,i}^{(1)} =-\frac{1}{\lambda_0} \widetilde{B}_i^{(N)} \, u_m h,\label{eq:Constant1}
\end{align}
where \(\widetilde{B}_i^{(N)}\in \mathbb{R}\) is a constant independent of \(u_m\), \(h\), \(t\), and \(x\), and is given by
\begin{align}
\widetilde{B}_i^{(N)} = \sum_{j=1}^N (\boldsymbol{C}_{N}^{-1})_{ij}. \label{eq:B_i_expression}
\end{align}

\paragraph{The $\mathcal{O}(\varepsilon^{1})$ approximation of the SWME}
\label{section2.3}

We substitute the solutions from \eqref{eq:trivial_solution} and \eqref{eq:MatrixC22} into the \(\mathcal{O}(\varepsilon)\)-order terms of Equation \eqref{eq:IntroSWME3}.
Solving for \(\alpha^{(2)}_{i,N}\) for all \(i = 1, \ldots, N\) yields the following closure relations for \(\boldsymbol{\alpha}_N^{(2)} \in \mathbb{R}^N\)
\begin{align}
    \widetilde{\boldsymbol{C}}_N\boldsymbol{\alpha}_N^{(2)} = -\frac{g}{4\nu_0\lambda_0}\big(\partial_x(h^4)\big) \widetilde{\boldsymbol{I}}^N_1 \boldsymbol{1} - \frac{1}{\lambda_0^2} u_mh^2 \widetilde{\boldsymbol{I}}^N_2\boldsymbol{1}, \label{eq:MatrixC3}
\end{align}
where $\widetilde{\boldsymbol{C}}_N\in\mathbb{R}^{N\times N}$, $\widetilde{\boldsymbol{I}}_1^N \in \mathbb{R}^{N\times N}$ and $\widetilde{\boldsymbol{I}}_2^N \in \mathbb{R}^{N\times N}$ are matrices with the entries for $i,j = 1,\ldots,N$ given by
\begin{align}
 (\widetilde{\boldsymbol{C}}_N)_{ij}&=(2i+1)C_{ij}, \\
    (\widetilde{\boldsymbol{I}}_1^N)_{ij} &= (\boldsymbol{C}_N^{-1})_{ij},\\
    (\widetilde{\boldsymbol{I}}_2^N)_{ij} &=(\boldsymbol{C}_N^{-1})_{ij}-(2i+1)\delta_{ij} \displaystyle\sum_{k,l=1}^N (\boldsymbol{C}_N^{-1})_{kl}.
\end{align}   
For \(i = 1, \dots, N\), rewrite Equation \eqref{eq:MatrixC3} as
\begin{align}
\alpha_{N,i}^{(2)} = -\frac{g}{4 \nu_0 \lambda_0} \big(\partial_x(h^4)\big)\widetilde{F}_i^{(N)} +  \frac{1}{\lambda_0^2} u_m h^2 \widetilde{D}_i^{(N)},\label{eq:Constant2}
\end{align}
where \(\widetilde{F}_i^{(N)} \in \mathbb{R}\) and \(\widetilde{D}_i^{(N)}\) are constants independent of \(u_m\), \(h\), \(x\), and \(t\). The explicit expressions for  \(\widetilde{F}_i^{(N)}\) and \(\widetilde{D}_i^{(N)}\) are
\begin{align}
\widetilde{F}_i^{(N)} &=  \sum_{j=1}^N (\widetilde{\boldsymbol{C}}_N^{-1} \boldsymbol{C}_N^{-1})_{ij}, \label{eq:F_i_expression}\\
\widetilde{D}_i^{(N)} &= - \sum_{j=1}^N (\widetilde{\boldsymbol{C}}_N^{-1} \boldsymbol{C}_N^{-1})_{ij} + \left( \sum_{k,l=1}^N (\boldsymbol{C}_N^{-1})_{k,l} \right) \left( \sum_{j=1}^N (\boldsymbol{C}_N^{-1})_{ij} \right). \label{eq:D_i_expression} 
\end{align}

\paragraph{The closed RSWME model for arbitrary $N$}
Substituting \eqref{eq:trivial_solution}, \eqref{eq:Constant1} and \eqref{eq:Constant2} into \eqref{eq:alphaNsol}, we obtain the closure relation for $\alpha_{N,j}$ for all $j= 1,\dots,N$ as
\begin{align}
        \alpha_{N,j} = -\frac{\varepsilon}{\lambda_0}\widetilde{B}_j^{(N)} u_m h\color{black} + \varepsilon^2\left( \frac{1}{\lambda_0^2}\widetilde{D}_j^{(N)} u_m h^2 -\frac{g}{4 \nu_0 \lambda_0} \widetilde{F}_j^{(N)} \partial_x(h^4)\right) + \mathcal{O}(\varepsilon^3).\label{eq:approxalphaj}
\end{align}
Substituting \eqref{eq:approxalphaj} into \eqref{eq:IntroSWME2} and ignoring the terms $\mathcal{O}(\varepsilon^3)$ and $\mathcal{O}(\varepsilon^4)$ produces a closed RSWME system that takes the following form
\begin{align}
    \partial_tU
    + A^N(U)\partial_xU = S^N(U).
    \label{eq:Closed_form_RSWME_sys}
\end{align}
The system matrix $A^N(U) \in \mathbb{R}^{2\times2}$ of \eqref{eq:Closed_form_RSWME_sys} is defined as
        \begin{align}
       A^N= \begin{pmatrix}
        0 & 1\\
        -u_m^ 2\left(1 - \frac{\varepsilon^2 h^2}{\lambda_0^2}\Gamma^N\right)+gh\left(1-\frac{\varepsilon^2 h^2}{\lambda_0\nu_0} \Phi^N\right)&  2u_m\displaystyle{\left(1+\frac{\varepsilon^2h^2}{\lambda_0^2}  \Gamma^N\right)}
    \end{pmatrix},\label{eq:Closed_form_RSWME}
        \end{align}
where the constants $\Gamma^N$ and $\Phi^N \in\mathbb{R}$ are given by
\begin{align}
    \Gamma^N = \displaystyle{ \sum_{j=1}^N \frac{1}{2j+1}}\left(\widetilde{B}_j^{(N)}\right)^2, \quad
    \Phi^N = \displaystyle{\sum_{j=1}^N \widetilde{F}_j^{(N)}}.\label{eq:GammaNPhiN}
\end{align} 
\\
The source term $S^N(U)\in\mathbb{R}^2$ in \eqref{eq:Closed_form_RSWME_sys} is 
\begin{align}
    S^N (U)=
    \begin{pmatrix}
        0 \\
        -\frac{\nu_0}{\lambda_0}u_m \left( 1 - \frac{\varepsilon}{\lambda_0} \Omega^N  h + \frac{\varepsilon^2}{\lambda_0^2} \Lambda^N h^2 \right)
    \end{pmatrix},\label{eq:sourcetermU}
\end{align}
where the constants $\Omega^N$ and $\Lambda^N \in\mathbb{R}$ are defined as
\begin{align}\Omega^N =\displaystyle{\sum_{i=1}^N \widetilde{B}_i^{(N)}},\quad\Lambda^N = \displaystyle{\sum_{i=1}^N \widetilde{D}_i^{(N)}}.\label{eq:OmegaNLambdaN}\end{align}
\\
We further present the following theorem regarding the RSWME of order \( N \).

\begin{theorem}[Identical closed RSWME systems]
\label{thm:identical}
    For \( N \geq 2 \), all RSWME systems are identical.
\end{theorem}

\begin{proof}
The detailed proof is provided in \ref*{s:appenB}.
\end{proof}

From Theorem \ref{thm:identical} we know that for all $N\geq2$ the system matrix $A^N \in \mathbb{R}^{2 \times 2}$ from \eqref{eq:Closed_form_RSWME} and the source term $S^N \in \mathbb{R}^{2}$ from \eqref{eq:sourcetermU} remain unchanged. As a result, using the numerically computed values of $h$ and $u_m$ from the RSWME2 system, the velocity profile \eqref{eq:velocity_expansion} can be obtained for arbitrary $N$ by substituting these solutions into the closure relations for $\alpha_{N,i}$, for $i=1,...,N$ from \eqref{eq:Constant2} in a post processing step. 

In theory, the RSWME achieve a substantial reduction in runtime relative to the SWME, which uses a much larger system of equations, scaling with dimension $N+2$.

\paragraph{Hyperbolicity of the RSWME for arbitrary \( N \)}
As a consequence of Theorem \ref{thm:identical}, the hyperbolicity conditions for arbitrary \( N \geq 2 \) are identical to those for the case \( N = 2 \), as specified by condition \eqref{eq:condition2}. For arbitrary \( N \), the model is hyperbolic, provided that the assumptions \( \lambda = \mathcal{O}(\varepsilon^{-1}) \) and \( \nu= \mathcal{O}(\varepsilon^{-1}) \) are satisfied, where $\varepsilon\ll1$ is a small scaling. Moreover, the hyperbolic regularisation approach for \( N = 2 \), as presented in Theorem \ref{thm:hyperbolicity_reg2}, can be extended to ensure global hyperbolicity for arbitrary \( N \). Hence, for arbitrary $N\geq2$, global hyperbolicity could be obtained by adding the regularisation term from \eqref{eq:regularisation2} to the left-hand side of \eqref{eq:Closed_form_RSWME_sys}.

\color{black}{
\paragraph{Consistency with the Navier-Stokes equations} Instead of working with the moment equations \eqref{eq:IntroSWME1}-\eqref{eq:IntroSWME3}, one could directly apply the asymptotic expansion methodology of this paper to the Navier-Stokes equations. The vertically resolved reference system in the coordinates $(x,t,\zeta)$ from \cite{Kowalski} that is derived from the Navier-Stokes equations is ideally suited for this. We will now briefly comment on the consistency of the previously derived models for $N=1$, $N=2$ and $N>2$ with the vertically resolved reference system and then derive the asymptotic expansion directly for the vertically-resolved reference system from \cite{Kowalski}, for which we note that the two boundary conditions are given by 
    \begin{align}
        \partial_\zeta u|_{\zeta=1}=\frac{h}{\lambda}u|_{\zeta=0},\quad \partial_\zeta u|_{\zeta=1}=0.\label{eq:BCs}
    \end{align}
    
Firstly, the $N=1$ model is \emph{not} consistent with the boundary conditions \eqref{eq:BCs} of the vertically-resolved reference system from \cite{Kowalski}. 
This can be directly seen substituting the \(\mathcal{O}(\varepsilon)\) moment approximation from equation \eqref{eq:final1_alpha1} into the vertical expansion \eqref{eq:velocity_expansion}, resulting in the velocity profile
\begin{align}
    u_{N=1}(\zeta) = u_m -\frac{hu_m\varepsilon}{4\lambda_0}(1-2\zeta). \label{eq:velocity_expansion11}
\end{align}

Secondly, however, the $N=2$ model \emph{is} consistent with the boundary conditions \eqref{eq:BCs} of the vertically-resolved reference system from \cite{Kowalski}. 
Similarly to the previous case, this can be observed by substituting the \(\mathcal{O}(\varepsilon)\) moment approximation from equations \eqref{eq:final2_alpha1}--\eqref{eq:final2_alpha2} into the vertical expansion \eqref{eq:velocity_expansion} which leads to the velocity profile
\begin{align}
    u_{N=2}(\zeta) = u_m -\frac{hu_m\varepsilon}{\lambda_0}\left( \frac{1}{2}\zeta^2 -\zeta+ \frac{1}{3}\right). \label{eq:velocity_expansion22}
\end{align}

Thirdly, the general $N>2$ model also fulfills the boundary conditions \eqref{eq:BCs} of the vertically-resolved reference system from \cite{Kowalski}. This is obvious from the closure relation \eqref{eq:Constant1} concluding that \(\alpha_{N,j}^{(1)} = 0\) for \(j > 2\) and \(N > 2\). Consequently, for \(N > 2\), the reconstructed \(\mathcal{O}(\varepsilon)\)-polynomial velocity remains parabolic and satisfies the underlying boundary conditions \eqref{eq:BCs} of the vertically-resolved reference system from \cite{Kowalski}. 

Finally, we can proceed differently and directly apply the asymptotic expansion methodology of this paper to the vertically-resolved reference system and substitute $u(x,\zeta,t)=u^{(0)}+\epsilon u^{(1)}+\mathcal{O}(\epsilon^2)$. The $\mathcal{O}(\varepsilon^{-1})$-term of the momentum equation lets us deduce the form of $u^{(0)}$ as follows
\begin{align}
   \frac{\nu_0}{h}\partial_{\zeta\zeta} u^{(0)} =0\Rightarrow u^{(0)}=C_0(x,t)+C_1(x,t)\zeta
\end{align}
and we propose a quadratic profile for $u^{(1)}$ as
\begin{align}
u^{(1)} = D_0(x,t) + D_1(x,t)\zeta + D_2(x,t)\zeta^2.
\end{align}
The coefficients $C_0, C_1, D_0, D_1, D_2$ are determined using the boundary conditions at the top and the bottom and the compatibility with the average velocity at the different scales:

\begin{enumerate}
    \item $\partial_{\zeta}u|_{\zeta=1}=0$ at $\mathcal{O}(1)$ $\Rightarrow$ $C_1 = 0$
    \item $\partial_{\zeta}u|_{\zeta=1}=0$ at $\mathcal{O}(\epsilon)$ $\Rightarrow$ $D_1 + 2D_2 = 0$
    \item $\partial_{\zeta}u|_{\zeta=0}=\frac{h\varepsilon}{\lambda_0}u|_{\zeta=0}$ at $\mathcal{O}(1)$ $\Rightarrow$ $C_1 = 0$
    \item $\partial_{\zeta}u|_{\zeta=0}=\frac{h\varepsilon}{\lambda_0}u|_{\zeta=0}$ at $\mathcal{O}(\epsilon)$ $\Rightarrow$ $D_1=\frac{h}{\lambda_0}C_0$
    \item $\int_0^1 u \, d\zeta = u_m $ at $\mathcal{O}(1)$ $\Rightarrow$ $C_0 = u_m$
    \item $\int_0^1 u \, d\zeta = u_m $ at $\mathcal{O}(\epsilon)$ $\Rightarrow$ $D_0 + \frac{D_1}{2} + \frac{D_2}{3}=0$
\end{enumerate}
This leads to a unique solution for the coefficients resulting in the profile
\begin{align}
 u(\zeta)= u_m -  \frac{h u_m \epsilon}{\lambda_0 }\left( \frac{1}{2} \zeta^2 - \zeta +  \frac{1}{3}\right) + \mathcal{O}(\epsilon^2),
\end{align}
which is the same as for the $N=2$ case \eqref{eq:velocity_expansion22}. The final two-equation macroscopic model, which is also equivalent to the $N=2$ model above, is then directly obtained by depth-averaging of the vertically-resolved reference system (omitted here for brevity, see details in \cite{Kowalski}). 

From this purely continuous solution the other results of this paper can be interpreted equivalently rooted directly in the vertically-resolved setting and the methodology of this paper can be seen as a systematic asymptotic analysis of the full Navier-Stokes equations, opening directions for future research.
}

\color{black}

\section{Numerical simulation}
In this section, we first present the time-integration scheme needed to handle the stiff source term for the viscous slip SWME.
We then present three numerical tests: a wave with sharp gradients (based on \cite{Kowalski}), a smooth sine wave, and an initial square root velocity profile. In turn, to examine the accuracy and computational cost of the RSWME and its closure relations for the moment variables, we conduct numerical experiments,
where the SWME are used as a reference solution.

\subsection{Numerical methods}
For large scalings of slip length $\lambda$ and viscosity $\nu$ as given by \eqref{eq:blah}, the source term $S(U)$ in the SWME system \eqref{eq:SystemPDE} is of order $\mathcal{O}(\varepsilon^{-1})$. This scaling can introduce stiffness, potentially leading to numerical instabilities. To avoid these numerical instabilities, operator splitting is employed in \cite{Huang2022}, allowing the source term to be integrated implicitly while maintaining stability in the numerical scheme. Then, following \cite{Huang2022}, the SWME model \eqref{eq:IntroSWME1}--\eqref{eq:IntroSWME3} is split into two parts as
\begin{align}
    \partial_tU+A(U)\partial_xU&=0,\label{eq:Systematrix}\\
    \partial_tU&=S(U).\label{eq:sourceterm}
\end{align}
To solve \eqref{eq:Systematrix}, we use a path-conservative finite volume scheme based on the polynomial viscosity methods developed for non-conservative products described in \cite{Canestrelli2009, CastroDaz2012}, with linear paths and polynomial viscosity matrix. For the source term in \eqref{eq:sourceterm}, an implicit backward Euler scheme is employed, resulting in
\begin{align}
    \frac{U^{n+1}-U^n}{\Delta t}=S(U^{n+1}), \label{eq:implicitsystem}
\end{align}
where \( U^n \) denotes the vector \( U \) evaluated at the discrete time \( t = t^n \), such that \( t^n \) corresponds to the \( n \)-th time step. Because $S$ is linear, the system \eqref{eq:implicitsystem} can be solved directly by inverting a matrix, similar to \cite{Huang2022}. {\color{black}As we are interested in the behaviour of transient solutions with flat bottom bathymetry, the chosen numerical scheme is not a priori designed to preserve steady state solutions, a property called well-balancing. However, one way to derive a well-balanced numerical scheme for the SWME or RSWME with varying bathymetry is to use the so-called global flux approximation, see \cite{Ciallela2026}.}

{\color{black} An underlying assumption of the viscous-slip RSWME model is that \( h \approx \mathcal{O}(1) \). In the following sections, initial conditions with \( h \approx 1 \) are selected, and \( h > 0 \) is observed in the simulations. While positivity-preserving numerical schemes are relevant as \( h \) approaches zero, they are not necessary for the test cases presented in our paper. For shallow water models, positivity preservation can be enforced using positivity limiters \cite{Ersing2025} or other methods \cite{Ciallella2022}.}

Under the assumption that $\lambda = \mathcal{O}(\varepsilon^{-1})$ and $\nu = \mathcal{O}(\varepsilon^{-1})$, the source terms corresponding to SWE \eqref{eq:SWEconv} and the RSWME \eqref{eq:Closed_form_RSWME_sys} are of order $\mathcal{O}(1)$. Consequently, these source terms do not exhibit stiffness, thereby permitting the choice of an explicit numerical scheme for the source term. In this work, the explicit scheme to solve \eqref{eq:sourceterm} within the SWE and RSWME model is given by
\begin{align}
\frac{U^{n+1} - U^n}{\Delta t} = S(U^n).
\end{align}

\subsection{Test I wave with a sharp height gradient}
To assess the accuracy of the RSWME1 model, we conduct a wave test with a sharp height gradient and compare the results with those obtained from the SWME1 model. The configuration of this test is provided in Table \ref{tab:test1}. For this analysis, we adopt the  wave test from \cite{Kowalski}, while modifying only the values of the viscosity \(\nu\) and slip length \(\lambda\) such that \eqref{eq:blah} holds. We evaluate the numerical results of the RSWME model for various values of \(\varepsilon \in \{0.01, 0.1, 1\}\). The smallest value, $\varepsilon = 0.01$, aligns with the equilibrium assumptions where $\varepsilon \ll 1$. The intermediate value, $\varepsilon = 0.1$, corresponds to a moderate deviation from equilibrium, while the largest value, $\varepsilon = 1$, represents a significant shift away from equilibrium, invalidating the assumption that $\varepsilon \ll 1$. Furthermore, in Equation \eqref{eq:blah}, we set \(\lambda_0 = 1\) and \(\nu_0 = 1\), and therefore in our simulations we have \(\lambda = \nu = \varepsilon^{-1}\).

For the spatial discretisation, we employ a PRICE scheme within the framework of polynomial viscosity methods \cite{Canestrelli2009, CastroDaz2012}. 
\begin{table}[t!]
\centering
\caption{Numerical setup for wave with sharp gradient simulation.}
\label{tab:test1}
\begin{tabular}{|l||c|}
\hline
viscosity            &  $\nu\in\{100,10,1\}$ \\
slip length           &$\lambda\in\{100,10,1\}$ \\
temporal domain & \( t\in[0,2] \) \\
spatial domain       & periodic \( x \in [-1, 1] \) \\
spatial resolution & $n_x=1000$\\
initial height             & $h(x,0) =  1 +\exp{(3\cos(\pi(x+0.5))-4.0)}$ \\
initial velocity profile        & \( u(x, 0, \zeta) = 0.5\zeta \) \\
CFL condition           & 0.7 \\
numerical scheme            & PRICE \cite{Canestrelli2009, CastroDaz2012} \\
\hline
\end{tabular}
\end{table}
In Figure \ref{fig:test1a}, the numerical solutions for the average velocity $u_m$ and height $h$ from the SWE, RSWME1, and SWME1 systems are presented. These numerical experiments were conducted for the various values of $\varepsilon$ over a time span of two seconds. 
For $\varepsilon = 0.01$ and $\varepsilon = 0.1$, the SWE, RSWME1, and SWME1 solutions are indistinguishable. In the limit as $\varepsilon \to 0$, the RSWME1 align with the SWE. According to \cite{Huang2022}, $\lambda \to \infty$ represents an equilibrium where the moment coefficients vanish. Since $\lambda \to \infty$ as $\varepsilon \to 0$, the moment variables in the SWME1 model become negligible for small $\varepsilon$. For the case $\varepsilon=1$, there are significant differences between the models. In particular, The RSWME1 provide a closer approximation to the average velocity $u_m$ of the SWME1 than the SWE do. 

The relative $L_1$-errors for the SWE and RSWME1 models, computed using the solution of the SWME1 as  reference, are given by
\begin{align*}
    E_{\text{SWE},h} &= \frac{\|h_{\text{SWE}} - h_{\text{SWME}}\|_1}{\|h_{\text{SWME}}\|_1}, &
    E_{\text{SWE},u_m} &= \frac{\|u_{m,\text{SWE}} - u_{m,\text{SWME}}\|_1}{\|u_{m,\text{SWME}}\|_1}, \\
    E_{\text{RSWME},h} &= \frac{\|h_{\text{RSWME}} - h_{\text{SWME}}\|_1}{\|h_{\text{SWME}}\|_1}, &
   E_{\text{RSWME},u_m} &= \frac{\|u_{m,\text{RSWME}} - u_{m,\text{SWME}}\|_1}{\|u_{m,\text{SWME}}\|_1}.
\end{align*}
In Table \ref{tab:ResTest1}, the RSWME1 yield a solution for height $h$ and average velocity $u_m$ that better approximates the SWME1 compared to the SWE for $\varepsilon \in \{0.01, 0.1, 1\}$. 
\begin{figure}[t!]
    \centering
    \includegraphics[width=\textwidth, keepaspectratio]{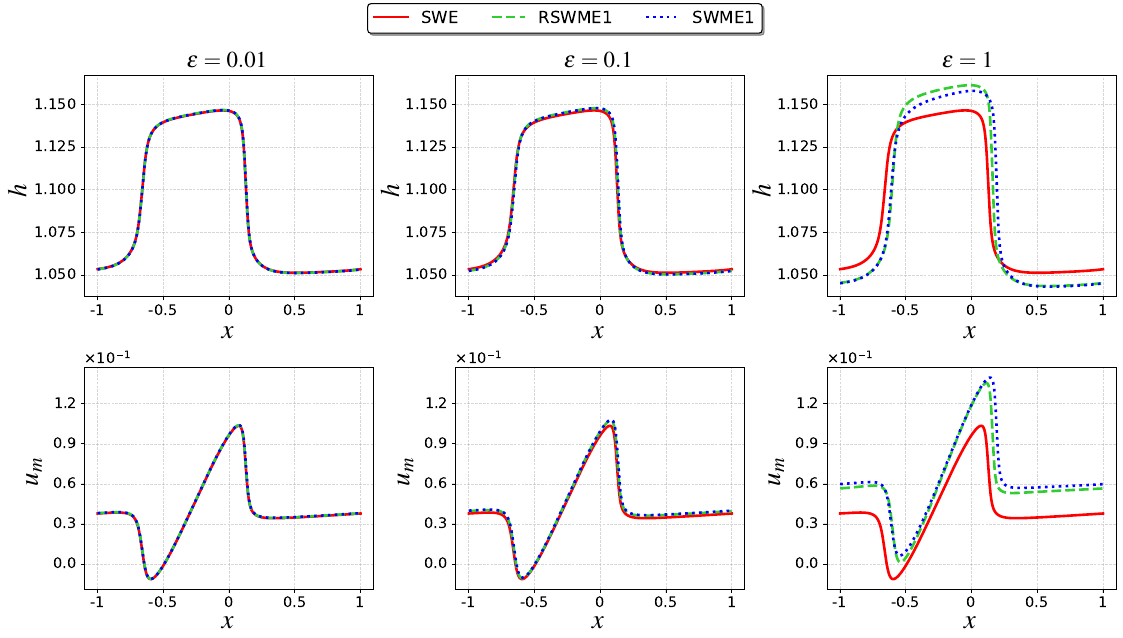}
    \caption{Wave with a sharp gradient test comparing the SWE (red, solid), RSWME1 (green, dashed), and SWME1 (blue, dotted) models at \(t=2.0\). The results are presented for the height \(h\) (top row) and velocity \(u_m\) (bottom row) across three values of \(\varepsilon\): \(\varepsilon = 0.01\) (left), \(\varepsilon = 0.1\) (middle), and \(\varepsilon = 1\) (right).}
    \label{fig:test1a}
\end{figure} 
By substituting the numerically computed solutions for $h$ and $u_m$ into the closure relation \eqref{eq:final1_alpha1}, we compute the coefficient $\alpha_1$. The closure relation in \eqref{eq:final1_alpha1} involves the spatial derivative of $h^4$ with respect to $x$, which we approximate numerically using a Godunov scheme. Experiments with various FVM schemes to discretise $\partial_x(h^4)$ revealed no significant differences. Thus, comparative analyses are omitted for conciseness.

In Figure \eqref{fig:Moment1a}, the numerical solution for $\alpha_1$ from the SWME1 model, as well as the approximated solution for $\alpha_1$ via the closure relation \eqref{eq:final1_alpha1}, are shown. One can observe that for $\varepsilon = 0.01$, both solutions converge to zero. For $\varepsilon = 0.1$, both the SWME1 and RSWME1 solutions equally deviate from zero, which indicates high accuracy of the RSWME1 model. However, for $\varepsilon = 1$, large peaks occur in the approximated solution via the RSWME1 solution. This can be explained as follows, the closure relation in \eqref{eq:final1_alpha1} includes a term of the form $-\varepsilon^2\frac{\partial_x(h^4)}{4g}$. When $\varepsilon = 1$ and the solution features sharp gradients in $h$, this term can lead to significant peaks in the resulting solution for $\alpha_1$. To investigate these deviations, as observed in Figure \ref{fig:Moment1a}, we examine the numerical derivative of $h^4$ with respect to $x$ for the SWE, RSWME, and SWME models. Figure \ref{fig:test1aDh} displays this derivative across different values of $\varepsilon$.
For $\varepsilon = 1$, both the SWME and RSWME yield significantly larger values of $\partial_x(h^4)$ compared to the SWE. Because this derivative is further scaled by $\varepsilon^2$, the resulting peaks are not damped as in the case $\varepsilon \ll 1$.
To observe for which cases these peaks do not occur in the reconstruction of the moments, we present a smooth case with a smaller gradient of $h$ in the next subsection.

\begin{figure}[t!]
    \centering
    \includegraphics[width=\textwidth, keepaspectratio]{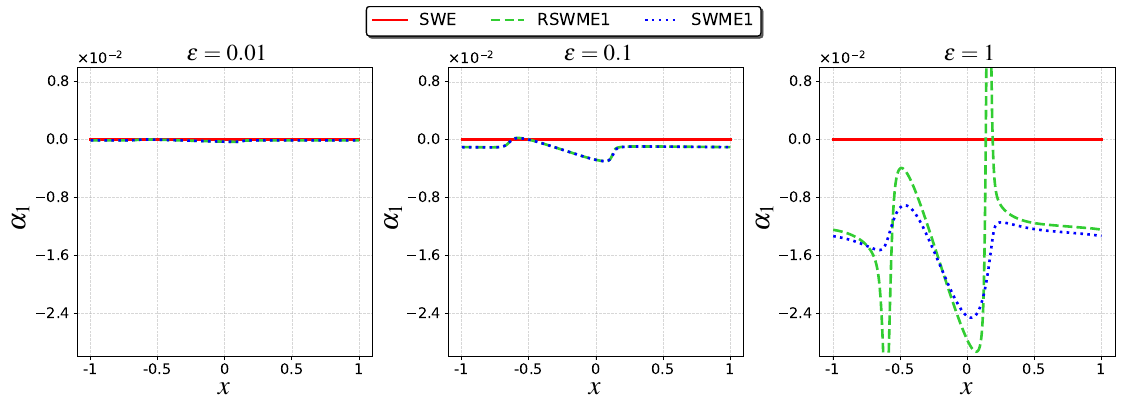}
    \caption{Wave with a sharp gradient test comparing the SWE (red, solid), RSWME1 (green, dashed), and SWME1 (blue, dotted) models at \(t=2.0\). The results are presented for the moment coefficient $\alpha_1$ across three values of \(\varepsilon\): \(\varepsilon = 0.01\) (left), \(\varepsilon = 0.1\) (middle), and \(\varepsilon = 1\) (right).}
    \label{fig:Moment1a}
\end{figure}

\begin{figure}[t!]
    \centering
    \includegraphics[width=\textwidth, keepaspectratio]{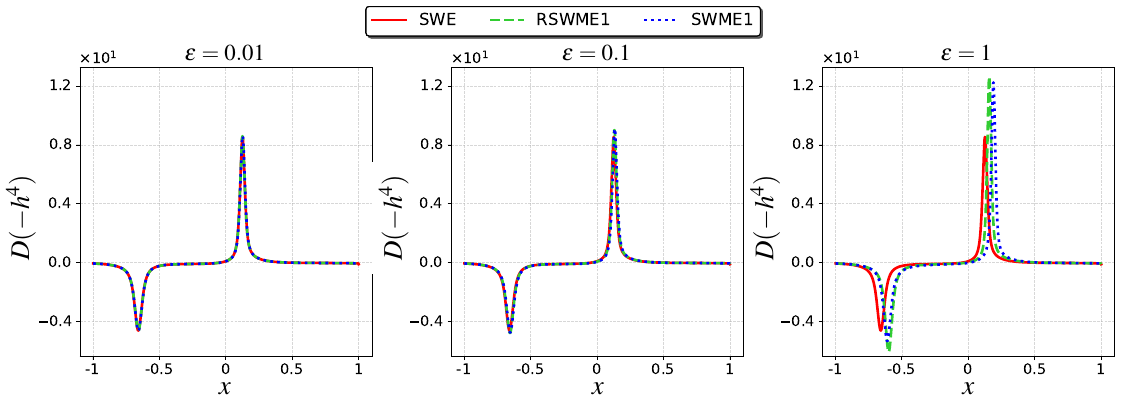}
    \caption{Wave with a sharp gradient test comparing the SWE (red, solid), RSWME1 (green, dashed), and SWME1 (blue, dotted) models at \(t=2.0\). The post processed $\partial_x(h^4)$-value computed with Godunov reconstruction is shown across three values of \(\varepsilon\): \(\varepsilon = 0.01\) (left), \(\varepsilon = 0.1\) (middle), and \(\varepsilon = 1\) (right).}
    \label{fig:test1aDh}
\end{figure}

\renewcommand{\arraystretch}{1.3}
\setlength{\aboverulesep}{0pt}
\setlength{\belowrulesep}{0pt}
\setlength{\cmidrulesep}{0pt}

\begin{table}[t!]
        \centering
          \caption{Relative $L_1$-error ($t=2.0$, $N=1$) for the wave with a sharp gradient.}
        \begin{tabular}{l|ccc}
            \toprule
            & ${\varepsilon} = 0.01$ & ${\varepsilon} = 0.1$ & ${\varepsilon} = 1$ \\
            \midrule
            SWE, in $h$ & 1.4600e-04 &1.2321e-03&1.0933e-02 \\
            RSWME1, in $h$ & 1.1144e-04 & 2.5440e-04&   2.9279e-03\\
            \midrule
            SWE, in $u$ & 7.3797e-03 & 5.6935e-02&  3.7623e-01 \\
            RSWME1, in $u$ &2.7112e-03& 1.0151e-02&6.4229e-02 \\
            \hline
        \end{tabular}
        \label{tab:ResTest1}
\end{table}

\subsection{Test II: smooth sine wave}

To investigate whether selecting an initial condition with smoother gradients leads to more accurate moment solutions via the closure relations, we present a smooth sine wave test case in Table 3. 
\begin{table}[h!]
\centering
\caption{Numerical setup for smooth sine wave simulation.}
\begin{tabular}{|l||c|}
\hline
kinematic viscosity            & $\nu\in\{100,10,1\}$ \\
slip length           &$\lambda\in\{100,10,1\}$ \\
temporal domain & \( t\in[0,2] \) \\
spatial domain       & periodic \( x \in [-1, 1] \) \\
spatial resolution & $n_x=1000$\\
initial height             & $h(x,0) =  1 -0.1\sin{(\frac{\pi x}{2})}^2$ \\
initial velocity profile         & \( u(x, 0, \zeta) = 0.5\zeta \) \\
CFL condition           & $0.7$ \\
numerical scheme            & PRICE \cite{Canestrelli2009, CastroDaz2012} \\
\hline
\end{tabular}\label{tab:test2}
\end{table} 
For this, we perform numerical tests for the smooth sine wave cases for the same values of $\varepsilon$ over a duration of two seconds. The simulations were focussed on the case where $N = 1$, employing the SWE, RSWME1, and SWME1. In Figure \ref{fig:test1b}, the numerical solutions for the average velocity $u_m$ and height $h$ are presented.
For $\varepsilon = 0.01$ and $\varepsilon = 0.1$, the SWE, RSWME1, and SWME1 solutions for the height $h$ are very similar. For $\varepsilon=1$, the RSWME1 yield a numerical solution that is closer to the SWME1 compared to the SWE for both height $h$ and average velocity $u_m$. Thus, the RSWME1 are more accurate than the SWE. 

In Table 4, the relative \(L_1\)-error for the RSWME1 and SWE solutions is shown with respect to the SWME1 as the reference solution. The results show that the RSWME1 provide a more accurate approximation of the SWME1 compared to the SWE. The smallest relative improvement of the RSWME1 over SWE when comparing their relative errors is 57\%, which is observed in the variable $u_m$ for $\varepsilon=0.01$. The largest relative improvement of 86\% occurs in variable $h$ for $\varepsilon=1$. \newline

In Figure \ref{fig:test1balpha} the numerical solutions for the moment variable $\alpha_1$ are shown. In this figure the erroneous peaks for $\varepsilon=1$ shown in the sharp gradient case from Figure \ref{fig:Moment1a} are not present. Hence, for a test case with smaller gradients of $h^4$ the closure relations for the moment variables are accurate up to a larger value of $\varepsilon$.

\begin{figure}[t!]
    \centering
    \includegraphics[width=\textwidth, keepaspectratio]{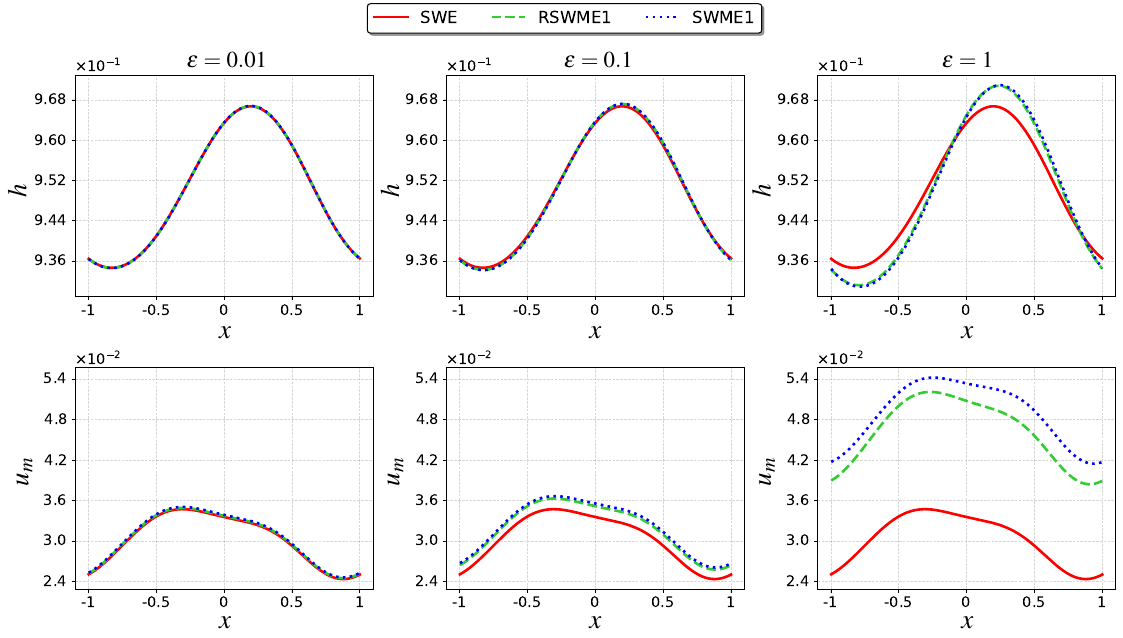}
    \caption{Smooth sine-wave test comparing the SWE (red, solid), RSWME1 (green, dashed), and SWME1 (blue, dotted) models at \(t=2.0\). The results are presented for the variables height $h$ (top row) and average velocity $u_m$ (bottom row) across three values of \(\varepsilon\): \(\varepsilon = 0.01\) (left), \(\varepsilon = 0.1\) (middle), and \(\varepsilon = 1\) (right).}
    \label{fig:test1b}
\end{figure}

\begin{figure}[t!]
    \centering
    \includegraphics[width=\textwidth, keepaspectratio]{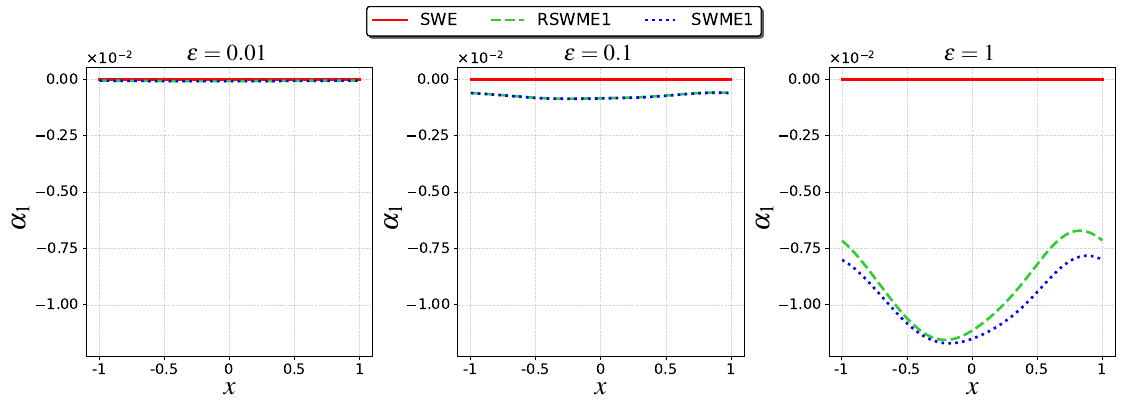}
    \caption{Smooth sine-wave test comparing the SWE (red, solid), RSWME1 (green, dashed), and SWME1 (blue, dotted) models at \(t=2.0\). The results are presented for the moment coefficient $\alpha_1$ across three values of \(\varepsilon\): \(\varepsilon = 0.01\) (left), \(\varepsilon = 0.1\) (middle), and \(\varepsilon = 1\) (right).}
    \label{fig:test1balpha}
\end{figure}

\begin{table}[b]
        \centering
        \caption{Relative $L_1$-error ($t = 2.0$ and $N = 1$) for the smooth sine wave test.}
        \begin{tabular}{l|ccc}
            \toprule
            & $\varepsilon = 0.01$ & ${\varepsilon} = 0.1$ & $\varepsilon = 1$ \\
            \midrule         
            SWE, in $h$ &    4.6365e-5 &3.6815e-4   &3.2456e-3 \\
            RSWME1, in $h$&   1.7695e-5& 5.4793e-5&   4.3920e-4 \\
            \midrule
            SWE, in $u_m$ &    9.4071e-3 & 5.8338e-2  &3.8118e-1  \\
            RSWME1, in $u_m$&   4.0201e-3&  1.0659e-2  &5.6338e-2\\
            \hline
        \end{tabular}
        \label{tab:ResTest2}
\end{table}

Now, for the case $N=2$, we compare the RSWME2 system from \eqref{eq: ASWME2} with the SWE and SWME2. Numerical tests for the smooth sine wave cases were again conducted for various values of $\varepsilon$ over a duration of two seconds. In Figure \ref{fig:testII}, the numerical solutions for the average velocity $u_m$ and height $h$ are presented. In Figure \ref{fig:testII} one can observe that the RSWME2 lead to a more accurate solution to the SWME2, compared to the SWE. In Table 5, the relative $L_1$-error for the RSWME2 and SWE solutions are shown with respect to the SWME2 as the reference solution. The errors in Table 4 and Table 5 are of a similar magnitude. This indicates that for both $N=1$ and $N=2$, the RSWME lead to an approximation closer to the SWME2 compared to the SWE. The smallest relative improvement of selecting the RSWME2 over the SWE is of 64\% and occurs in variable $u_m$, when $\varepsilon=0.01$. The largest relative improvement takes place when $\varepsilon=0.1$ for 88\% and is measured in $h$.
\begin{figure}[h!]
    \centering
    \includegraphics[width=\textwidth, keepaspectratio]{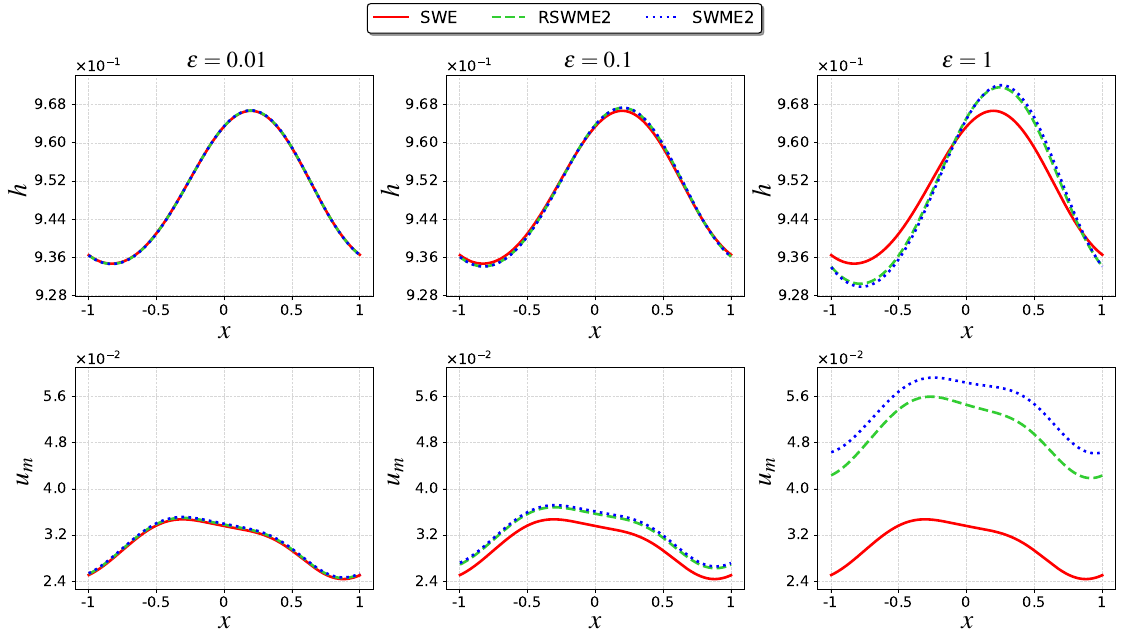}
    \caption{Smooth sine-wave test comparing the SWE (red, solid), RSWME2 (green, dashed), and SWME2 (blue, dotted) models at \(t=2.0\). The results are presented for the variables height $h$ and average velocity $u_m$ across three values of \(\varepsilon\): \(\varepsilon = 0.01\) (left), \(\varepsilon = 0.1\) (middle), and \(\varepsilon = 1\) (right).}
    \label{fig:testII}
\end{figure}

By substituting the numerical approximations of $u_m$ and $h$ into the closure relations \eqref{eq:final2_alpha1} and \eqref{eq:final2_alpha2}, we reconstruct the moment variables $\alpha_1$ and $\alpha_2$ in the RSWME2 system. The $\partial_x(h^4)$-term in \eqref{eq:final2_alpha1}--\eqref{eq:final2_alpha2} is numerically computed using a Godunov FVM scheme as before. In Figure \ref{fig:testIIalpha}, the moment variables from the RSWME2 framework are presented alongside the numerically computed SWME2 moment variables.
For $\varepsilon = 0.01$, the values of $\alpha_1$ and $\alpha_2$ in both the RSWME2 and SWME2 frameworks vanish and remain close to zero. For $\varepsilon = 0.1$, the values of $\alpha_1$ and $\alpha_2$ deviate further from zero compared to smaller values of $\varepsilon$. Notably, for $\varepsilon = 0.1$, the numerical approximations of $\alpha_1$ and $\alpha_2$ coincide for both models, RSWME2 and SWME2.
For $\varepsilon = 1$, the approximated moments $\alpha_1$ and $\alpha_2$ for the RSWME2 lie between the curves related to SWME2 and the zero line corresponding to the SWE. Additionally, for $\varepsilon = 0.1$ and $\varepsilon = 1$, the approximated moment $\alpha_2$ related to SWME2 and RSWME2 are consistently smaller than those for $\alpha_1$.

\begin{figure}[t!]
    \centering
    \includegraphics[width=\textwidth, keepaspectratio]{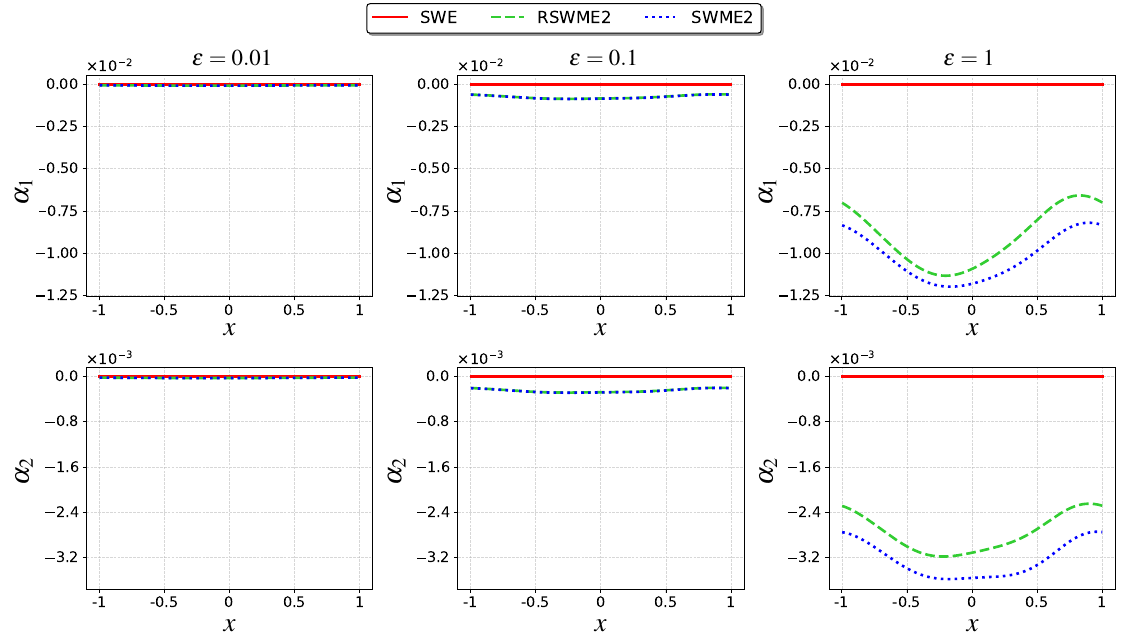}
    \caption{Smooth wave test comparing the SWE (red, solid), RSWME2 (green, dashed), and SWME2 (blue, dotted) models at \(t=2.0\). The results are presented for the moment coefficients $\alpha_1$ and $\alpha_2$ across three values of \(\varepsilon\): \(\varepsilon = 0.01\) (left), \(\varepsilon = 0.1\) (middle), and \(\varepsilon = 1\) (right).}
    \label{fig:testIIalpha}
\end{figure}

\begin{table}[t!]
        \centering
         \caption{Relative $L_1$-error ($t=2.0$, $N=2$) for the smooth sine wave test.}
        \begin{tabular}{l|ccc}
            \toprule
            & ${\varepsilon} = 0.01$ & ${\varepsilon} = 0.1$ & ${\varepsilon} = 1$ \\
            \midrule
            SWE, in $h$ & 5.6009e-5   & 4.7676e-4& 4.1328e-3\\
            RSWME2, in $h$&1.6798e-5 &   5.5410e-5& 6.1433e-4\\
            \midrule
            SWE, in $u_m$ & 1.1182e-2   & 7.2932e-2  & 4.3821e-1\\
            RSWME2, in $u_m$&4.0137e-3&    1.0549e-2 &  7.4865e-2\\
            \hline
        \end{tabular}
        \label{tab:ResTest3}
\end{table}

\paragraph{Comparison of RSWME2 and HRSWME2}
Using the configuration of Table 3, we perform numerical tests for $N=2$ comparing the RSWME2 and the HRSWME2 from Theorem \ref{thm:hyperbolicity_reg2} for various values of $\varepsilon=0.1,1,10$. In Table 6, the errors for the numerical tests are presented, with the SWME2 serving as the reference solution.  For the values of $\varepsilon = 0.1$ and $\varepsilon = 1$, we can observe that the differences in the errors between the RSWME2 and HRSWME2 are not significant. There are two reasons for this. First, the RSWME2 model is stable, since $\varepsilon$ is not too large and thereby the condition \eqref{eq:condition2} holds. Second, because $h < 1$ holds in Figure \ref{fig:testII}, the hyperbolicity condition $
    0 < h < 1 < \frac{3\sqrt{5}}{\varepsilon},
$ is satisfied for both $\varepsilon = 0.1 $ and $\varepsilon = 1$.

In Theorem \ref{thm:hyperbolicity_reg2}, the added hyperbolic regularisation term is of order $\mathcal{O}(\varepsilon^4)$. For $\varepsilon = 0.1$, this term is negligible due to its small magnitude. Furthermore, the friction term is scaled by a factor of $\frac{4}{6 \cdot 180^2} \approx 2 \times 10^{-5}$, which means that its contribution is minimal. Consequently, even for $\varepsilon = 1$, the additional term remains small.

Given that the added regularisation term is negligible for both $\varepsilon = 0.1$ and $\varepsilon = 1$, the difference between the solutions for 
RSWME2 and HRSWME2 can be expected to be insignificant. On the other hand, the case with the largest value of $\varepsilon = 10$, the numerical solutions to the HRSWME2 exhibit smaller errors compared to those related to RSWME2. Importantly, when $\varepsilon = 10$, the condition in \eqref{eq:condition2} no longer holds. This potential breakdown of hyperbolicity of the RSWME2 can introduce instabilities, resulting in erroneous solutions. Consequently, the inclusion of a hyperbolicity term in the HRSWME2 helps to mitigate these instabilities, which explains its improved accuracy. Furthermore, for $\varepsilon=10$, the errors related to the SWE are smaller than those related to both the HRSWME2 and RSWME2. This can be explained by the fact that a core model assumption no longer holds, which states that $\varepsilon$ is small. Hence, the HRSWME2, while hyperbolic in comparison to the RSWME2, are no longer valid for $\varepsilon=10$.

In further tests we will not continue to use the HRSWME2 model since for valid values of $\varepsilon$ we see no improvement in accuracy or in stability. Whereas accuracy gains are observed for large $\varepsilon$ values, the underlying assumptions of the RSWME2 model break down. In this situation, we also observe that larger errors for the HRSWME2 are produced in comparison to the SWE model. Consequently, introducing a hyperbolic regularisation term to the HRSWME2 does not yield any benefit for the numerical solution.

\begin{table}[t!]
        \centering
         \caption{The relative $L_1$-error ($t=2.0$, $N=2$) for the sine wave test.}
        \begin{tabular}{l|ccc}
            \toprule
            & ${\varepsilon} = 0.1$ & ${\varepsilon} = 1$ & ${\varepsilon} = 10$ \\
            \midrule
            SWE, in $h$  &   4.7676e-4 & 4.1328e-3&1.6712e-2\\
            RSWME2, in $h$  & 5.5410e-5& 6.1433e-4& 3.5656e-1\\
            HRSWME2, in $h$ & 5.5410e-5 &6.1407e-4& 3.6133e-2\\
            \midrule
            SWE, in $u_m$ &    7.2932e-2 & 4.3821e-1& 8.2080e-1\\
            RSWME2, in $u_m$&   1.0549e-2 & 7.4865e-2& 1.5482e-0\\
            HRSWME2, in $u_m$ & 1.0549e-2 & 7.4790e-2& 1.0000e-1\\
            \hline
        \end{tabular}
        \label{tab:ResTest4}
\end{table} 
\newpage
\subsection{Test III: square root velocity profile}

To investigate whether increasing the number of moments results in a more accurate approximation, we select an initial velocity profile from \cite{Koellermeier2022} with a square root $u(x,\zeta,t)=\frac{2}{3}\sqrt{\zeta}$. The initial conditions for the average velocity and the first moment coefficients are
\begin{alignat*}{3}
    u_m(x,0) &= 1, &\quad \alpha_1(x,0) &= -\frac{3}{5},&\quad \alpha_2(x,0) &= -\frac{1}{7},\\
    \alpha_3(x,0) &= -\frac{1}{15}, &\quad \alpha_4(x,0) &= -\frac{3}{77},&\quad \alpha_5(x,0) &= -\frac{3}{35},\quad \alpha_6(x,0) = \frac{1}{55}.
\end{alignat*}
As in the previous examples, we provide the configurations for this example in Table 7. 
\begin{table}[h!]
\centering
\label{tab:test3}
\caption{Numerical setup for the square root velocity test.}
\begin{tabular}{|l||c|}
\hline
kinematic viscosity            & $\nu=2$ \\
slip length           &$\lambda=2$ \\
temporal domain & \( t\in[0,2] \) \\
spatial domain       & periodic \( x \in [-1, 1] \) \\
spatial resolution & $n_x=1000$\\
initial height             & $h(x,0) =  1 -0.1\sin{(\frac{\pi x}{2})}^2$ \\
initial velocity profile          & \( u(x, 0, \zeta) =\frac{3}{2}\sqrt{\zeta} \) \\
CFL condition           & $0.7$ \\
numerical scheme            & PRICE \cite{Canestrelli2009, CastroDaz2012} \\
\hline
\end{tabular}
\end{table}
We perform numerical tests for $N=2,4,6$ to compare the accuracy for the SWME, RSWME and SWE solutions. It is important to note that according to Theorem \ref{thm:identical}, the reduced models RSWME2, RSWME4 and RSWME6 are identical. 
In Figure \ref{fig:test3a}, the numerical solutions for the average velocity $u_m$ and height $h$ are presented. For every choice of $N=2,4,6$, the RSWME model is closer to the the SWME model for both $h$ and $u_m$ compared to the SWE model. 

\begin{figure}[t!]
    \centering
    \includegraphics[scale=0.75]{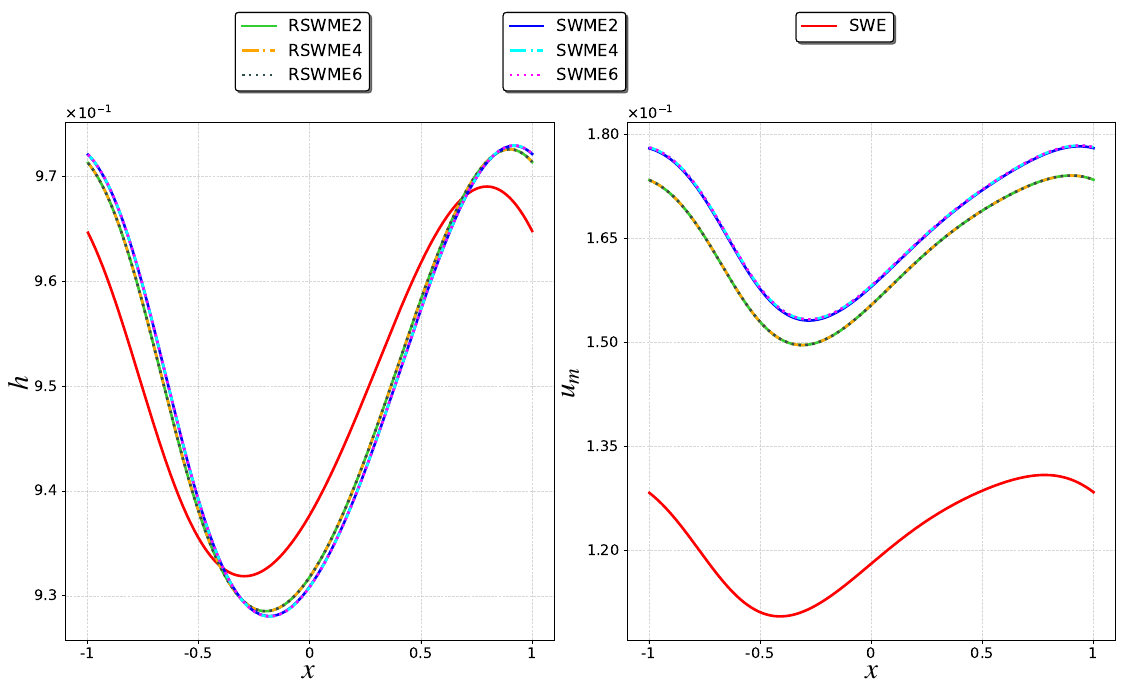}
    \caption{Square root velocity profile test comparing the SWE, RSWME, and SWME models for \(N=2\) (solid), $N=4$ (dashed) and $N=6$ (dotted) at \(t=2.0\). The results are presented for the variables height $h$ (left) and average velocity $u_m$ (right) and \(\varepsilon=0.5\).}
    \label{fig:test3a}
\end{figure}

As in the previous examples, the moment variables $\alpha_{N,i}$, $i=1,\ldots,N$ are computed using the closure relation \eqref{eq:approxalphaj} and the 
numerical approximations of $u_m$ and $h$. The $\partial_x(h^4)$-term in \eqref{eq:approxalphaj} is again numerically computed using a Godunov FVM scheme.  In Figure \ref{fig:test3b}, the moment variables from the RSWME framework are presented alongside the numerically computed moment variables related to the SWME. We observe that the magnitude of $\alpha_{N,i}$ for $i=1,\ldots,N$, decreases strictly as we increase the order of the moments. This makes sense, because the $\alpha_{N,i}$, for $i=1,\ldots, N$, are expansion coefficients of a converging expansion.

Computations for $\alpha_5$ and $\alpha_6$ show that the coefficients for the RSWME are already equal to zero, which goes into the same direction as the SWME, for which a small contribution of the size of $\mathcal{O}(10^{-7})$ to $\mathcal{O}(10^{-8})$ remains.
Furthermore, one can observe that for \( N = 2, 4, 6 \), the RSWME moment closure relations yield solutions for \( \alpha_{N,i} \) (where \( i \leq 4 \)) that provide a more accurate approximation of the SWME solutions compared to the case where \( \alpha_{N,i} = 0 \), for $i=1,\ldots,N$ in the SWE framework. 
\begin{figure}[t!]
    \centering
    \includegraphics[scale = 0.75]{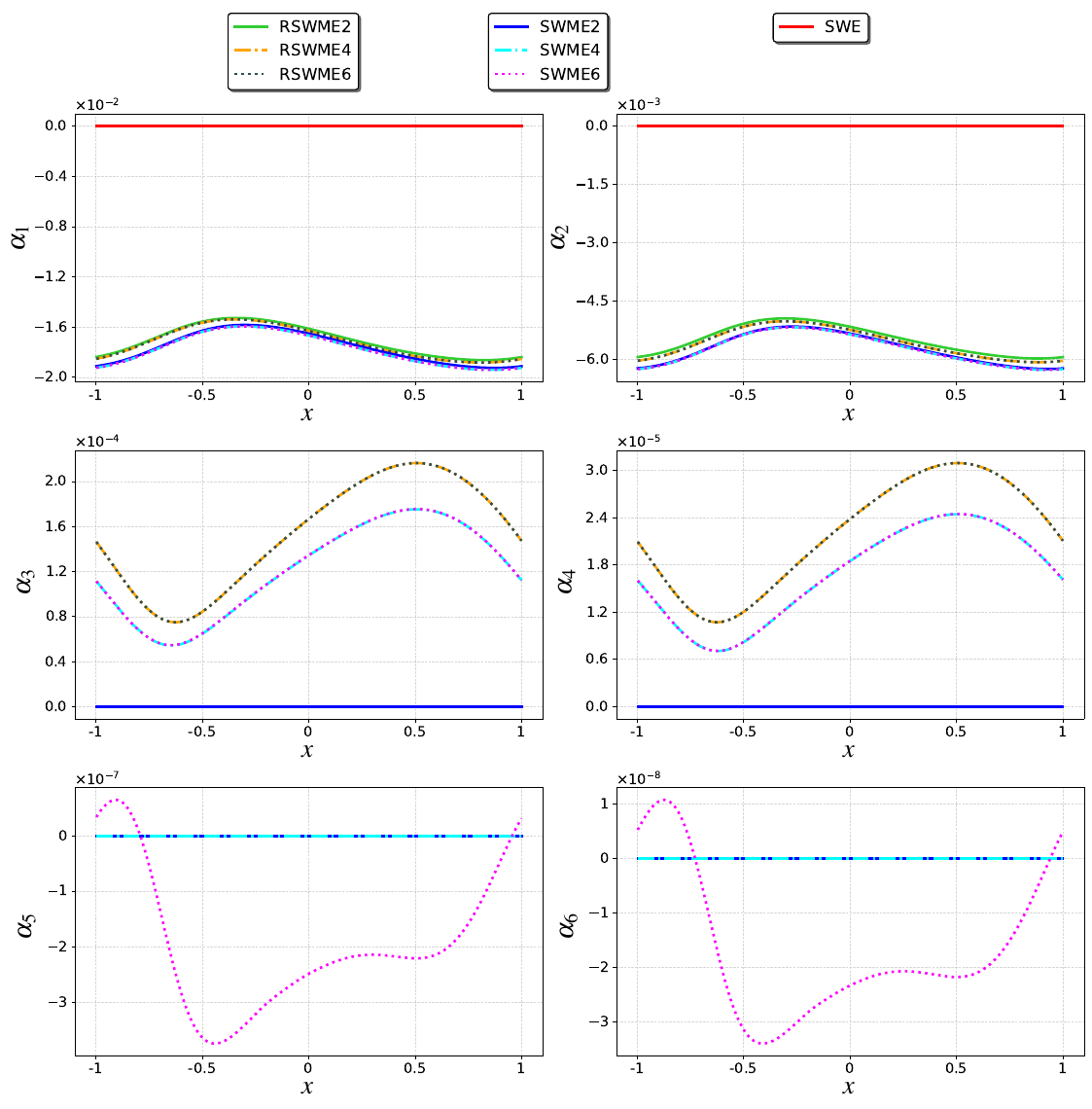}
    \caption{Square root velocity profile test comparing the SWE, RSWME, and SWME models for \(N=2\) (solid), $N=4$ (dashed) and $N=6$ (dotted) at \(t=2.0\). The results are presented for the moment coefficient $\alpha_{N,i}$, $i=1,\ldots,6$ and \(\varepsilon=0.5\).}
    \label{fig:test3b}
\end{figure} 
Furthermore, by employing the numerically computed solutions for \( u_m \) and the moment variables, the vertical velocity profile at \( x = x_0 = 0.5 \) can be visualised using the relation provided in Equation \eqref{eq:velocity_expansion}. In Figure \ref{fig:test3c}, we compare the velocity profiles for the SWE, SWME and RSWME for $N=2,4,6$, with the fixed constant $\varepsilon=0.5$. The vertical velocity profile predicted by the RSWME aligns more closely with that of the SWME than does the profile obtained from the SWE. For every choice of $N=2,4,6$ the RSWME approximates the SWME better, compared to the SWE. 

\begin{figure}[t!]
    \centering
    \includegraphics[scale=0.75]{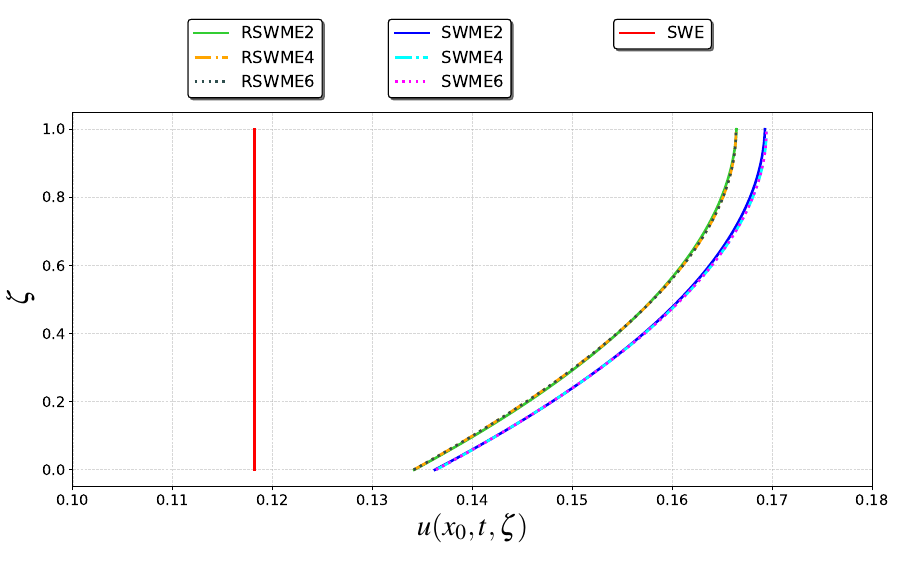}
    \caption{Square root velocity profile test comparing the SWE, RSWME, and SWME models for \(N=2\) (solid), $N=4$ (dashed) and $N=6$ (dotted) at \(t=2.0\). The results are presented for the vertical velocity profile $u(x_0,t,\zeta)$ for \(\varepsilon=0.5\) and $x_0=0.0$.}
    \label{fig:test3c}
\end{figure}

In Table \ref{tab:runtime}, the computational runtimes for the various models are presented. The differences in runtimes between the SWE and RSWME are comparatively negligible. However, we can observe that for $N=2,4,6$, the runtime for the SWME is always significantly larger compared to the SWE and RSWME. This can be explained by the fact that the SWE and RSWME solve a system of only $2$ variables, whereas the SWME solve a system of $N+2$ variables. 
The small difference in runtime between the RSWME and SWE arises from an additional post-processing step required for the RSWME: computing the closure relations for the moment variables \eqref{eq:approxalphaj}. In the numerical test, results are computed up to $t=2$, and the closure relations for the moment variable are evaluated only once at the final time step.

Our experimental results in Table \ref{tab:runtime} indicate that the RSWME model is between 18\% (for $N=2$) and 77\% (for $N=6$) faster than the SWME. The simulation time of the RSWME2 exceeds the runtime of the SWE by a mere 5\%, which is an insignificant amount of overhead. Furthermore, our results show that increasing the number of moments significantly raises the simulation time for the SWME model, whereas the simulation time for the RSWME remains unchanged. 

For conciseness, the runtimes for the test cases with a sharp wave and a smooth sine wave are not included. During experiments, we noticed similar speedups as in Table \ref{tab:runtime}.

\begin{table}[t!]
    \centering
     \caption{Comparison of CPU runtime for simulating the square root velocity profile test case up to \( t = 2 \) on AMD EPYC 7763 64-core processor.}
    \begin{tabular}{|c||c|c|c|}
    \hline
        \textbf{\( N \)} & \textbf{SWE (s)} & \textbf{SWME (s)} & \textbf{RSWME (s)} \\ \hline \hline  
        0& 132.6 & - & - \\ \hline
        2& - & 170.0 &138.7\\ \hline
         4&-&307.4& 138.6\\ \hline
         6&-&606.8&138.9\\ \hline
    \end{tabular}
      \label{tab:runtime}
\end{table}

\section{Conclusion}
In this paper we introduced the Reduced Shallow Water Moment Equations (RSWME), as an accurate but more efficient approximation of Shallow Water Moment Equations (SWME) with high viscosity and large slip length. The RSWME are able to model vertical velocity profiles with slight deviations from a uniform velocity profile, while having fewer variables compared to the SWME. 

A notable cost of the SWME lies in its need to include a larger number of variables compared to the Shallow Water Equations (SWE), even when the SWME have an equilibrium that could be represented by fewer variables. To simplify the SWME near an equilibrium where viscosity and slip length are large, we apply an asymptotic expansion to analyse small deviations around this equilibrium.

The SWME were subjected to a large-viscosity and large-slip scaling. Applying asymptotic analysis to this SWME model yielded explicit closure relations for moment variables up to second order in the smallness parameter. The RSWME were obtained by substituting these closure relations into the original SWME model. Because the RSWME are a dimensionally reduced model and contain fewer variables, the RSWME have a lower computational cost. Additionally, we showed that the closed RSWME system is identical for large numbers of moments.

The simple and closed form allowed for an in-depth analysis of the equations, which was exemplified by studying its hyperbolicity. Similar to related SWME models, a hyperbolic regularisation was performed.

The accuracy and computational efficiency of the RSWME were evaluated through numerical tests involving a wave with a sharp height gradient, a smooth sine wave test case, and a square root initial velocity distribution.
While showing remarkable accuracy for the height and average velocity, the test with a wave and a sharp gradient of height yields incorrect solutions in the moment variables for large deviations from equilibrium. For smooth sine waves, accurate RSWME solutions were obtained that approximate the SWME more accurately compared to the SWE, for both small and large deviations from equilibrium. In the smooth sine wave test the relative error of the RSWME improves by up to 88\% relatively compared to the SWE. A test with a square root velocity profile showed that the RSWME yield a more accurate velocity profile than the SWE. For the square root velocity profile test, the RSWME model demonstrates a runtime reduction up to 77\%, compared to the SWME.

These findings support two distinct interpretations: (1) The RSWME achieve accuracy comparable to the SWME while requiring significantly lower computational time, and; (2) The RSWME model outperforms the SWE in accuracy with similar computational costs. These outcomes render the RSWME model efficient for simulations.

For future work, the application of asymptotic approximations to SWME with alternative scalings for viscosity and slip length are of interest. Such an analysis could focus on the different equilibria detailed in Section~\ref{sec:alternative:scalings}, for example, by developing reduced models characterised by a small slip length.

\section*{CRediT authorship contribution statement}

\textbf{M. Daemen:} Conceptualization, Methodology, Software, Formal Analysis, Investigation, Data Curation, Writing - Original Draft, Writing - Review \& Editing, Visualization. \textbf{J. Careaga:} Conceptualization, Writing - Review \& Editing, Supervision, Visualization. \textbf{Z. Cai} Conceptualization, Resources, Writing - Review \& Editing. \textbf{J. Koellermeier:} Conceptualization, Resources, Writing - Review \& Editing, Supervision, Funding acquisition.

\section*{Acknowledgments}
This work is supported by the Dutch Research Council (NWO) through the ENW Vidi project HiWAVE with file number VI.Vidi.233.066. We acknowledge Rik Verbiest for providing a numerical simulation code for the Shallow Water Moment Equations (SWME). We also thank the anonymous reviewers for their insightful and constructive comments.

\section*{Declaration of generative AI and AI-assisted technologies in the manuscript preparation process.}
 During the preparation of this work the author(s) used Le Chat from Mistral AI to improve writing style, improve formatting of equations and debug software. After using this tool/service, the author(s) reviewed and edited the content as needed and take(s) full responsibility for the content of the published article.

\section*{Data availability}
Data will be made available on request.

\appendix
\section*{Appendix A. Derivations for the $\mathcal{O}(\varepsilon^2)$-closure relation}
\label{A:Prelim}In preliminary work conducted by D.~Wang \cite{Wang2023}, closure relations for the asymptotic components \(\smash{\alpha_1^{(0)}}\), \(\smash{\alpha_1^{(1)}}\), and \(\alpha_2^{(2)}\) in \eqref{eq:expansion_alpha1} were derived for the case \(N = 1\). In this appendix, we detail the derivation steps for the closure relation of $\alpha_2^{(2)}$.
Extracting the $\mathcal{O}(\varepsilon)$ terms from the final equation of \eqref{eq:SWME_viscous_system} and substituting $\alpha_1^{(1)} = -\frac{u_mh}{4\lambda_0}$ yields
\begin{equation}
    \label{eq:momentum_correction2A}
    \partial_t (h^2 u_m) + 2 \partial_x (h^2 u_m^2) = u_m \partial_x (h^2 u_m) - 3 \frac{\nu_0}{\lambda_0} \left( u_m h - \frac{16 \lambda_0^2}{h} \alpha_1^{(2)} \right).
\end{equation}
To simplify \eqref{eq:momentum_correction2A} we use
\begin{equation}
    \label{eq:momentum_balanceAppendix}
    \partial_t (h^2 u_m) = -\frac{1}{2} \partial_x (h^2 u_m^2) - h \left[ \partial_x \left( h u_m^2 + \frac{g h^2}{2} \right) + \frac{\nu_0}{\lambda_0} u_m \right].
\end{equation}
Now, we present the underlying derivation steps to obtain \eqref{eq:momentum_balanceAppendix}.
Applying the chain rule to the left-hand side of Equation~\eqref{eq:momentum_balanceAppendix} yields
\begin{equation}
    \label{eq:momentum_balanceAppendixI}
    \partial_t (h^2 u_m) =hu_m \partial_t(h) +h\partial_t(hu_m).
\end{equation}
Then, substituting the mass-conservation equation given by the first line of~\eqref{eq:SWME_viscous_system}, and the momentum-balance equation given by the second line of~\eqref{eq:SWME_viscous_system}, we have
\begin{equation}
    \label{eq:momentum_balanceAppendixII}
    \partial_t (h^2 u_m) =-hu_m \partial_x(hu_m) -h\partial_x\left(hu_m^2+ \frac{1}{3}\alpha_1^2+g\frac{h^2}{2}\right)-h\frac{\nu_0}{\lambda_0}\left(u_m+\alpha_1\right).
\end{equation}
Given that $\alpha_1 = \mathcal{O}(\varepsilon)$, we omit the $\alpha_1$ term in~\eqref{eq:momentum_balanceAppendixII}. Applying the chain rule to the first term on the right-hand side of this equation yields~\eqref{eq:momentum_balanceAppendix}. Substituting this into Equation~\eqref{eq:momentum_correction2A} yields the following $\mathcal{O}(\varepsilon)$-component of \eqref{eq:SWME_viscous_system} with no time-derivative term
\begin{align}
 \frac{3}{2} \partial_x (h^2 u_m^2) - h \left[ \partial_x \left( h u_m^2 + \frac{g h^2}{2} \right) + \frac{\nu_0}{\lambda_0} u_m \right]   = u_m \partial_x (h^2 u_m) - 3 \frac{\nu_0}{\lambda_0} \left( u_m h - \frac{16 \lambda_0^2}{h} \alpha_1^{(2)} \right).\label{eq:simplification1A1}
\end{align}
Additionally, to further simplify \eqref{eq:simplification1A1}, we employ the relation
 \begin{align}
     \frac{3}{2} \partial_x (h^2 u_m^2) = h\partial_x \left( h u_m^2  \right) + u_m \partial_x (h^2 u_m)= 0.\label{eq:simplification2A1}
 \end{align}
Finally, the closure relation for $\alpha_1^{(2)}$ is obtained by substituting \eqref{eq:simplification2A1} into \eqref{eq:simplification1A1} as
\begin{equation}
    \label{eq:alpha1_finalS}
    \alpha_1^{(2)} = \frac{1}{48 \nu_0 \lambda_0} \left( -\frac{g}{4} \partial_x (h^4) + \frac{2 \nu_0}{\lambda_0} u_m h^2 \right).
\end{equation}

\section*{Appendix B. Proof of Theorem \ref{thm:identical}}
\label{s:appenB}
In this section, we prove Theorem \ref{thm:identical}. First, we present Lemma \ref{lem:lem1} that explains how $\boldsymbol{C}_N\boldsymbol{b}=1$, can be solved for arbitrary $N$. Using Lemma \ref{lem:lem1}, Theorem \ref{thm:identical} is proven.

\color{black}

\begin{lemma}
\label{lem:lem1}
Let \(N \geq 2\) and \(\boldsymbol{C}_N\in \mathbb{R}^{N\times N  }\) be the matrix with entries given by the constants $\boldsymbol{C}_{Nij}=C_{ij}$ for $i,j = 1,\dots,N$ defined in \eqref{eq:Cconst}, and let $\boldsymbol{1} \in \mathbb{R}^N $ be the vector of ones. Then, the following linear system
\begin{align}
 \boldsymbol{C}_N \boldsymbol{b}_N = \boldsymbol{1},\label{eq:lemma}
\end{align}
has a unique solution \(\boldsymbol{b}_N = (b_1, b_2, \ldots, b_N)^{\rm T}\in \mathbb{R}^N\) given by
\[
b_1 = \frac{1}{4}, \quad b_2 = \frac{1}{12}, \quad \text{and} \quad b_i = 0 \quad \text{for } i = 3, \ldots, N.
\]
\end{lemma}

\begin{proof}
We first observe that the matrix $\boldsymbol{C}_N$ can be computed component-wise as
\begin{align}
\begin{aligned}
\boldsymbol{C}_{Nij}  =C_{ij}=
\begin{cases}
0 & \text{if } i - j \text{ is odd}, \\
2\min{(i,j)}(\min{(i,j)}+1) & \text{if } i - j \text{ is even},
\end{cases}
\end{aligned}\label{eq:DefforC}
\end{align}
for $i,j = 1,\ldots,N$, according to \cite{Huang2022}.
From \eqref{eq:DefforC} we can readily observe that its first and second columns of $\boldsymbol{C}_N$ are the respective vectors
\begin{align*}
 \boldsymbol{c}_1 &= (4,0,4,0,4,\dots)^{\rm T}\in \mathbb{R}^N,\\
 \boldsymbol{c}_2 &= (0,12,0,12,0,\dots)^{\rm T}\in \mathbb{R}^N.
\end{align*}
Moreover, multiplying matrix $\boldsymbol{C}_N$ times $\boldsymbol{b}_N = (1/4, 1/12,0,0,\dots,0)^{\rm T}\in \mathbb{R}^N$ gives
\begin{align*}
 \boldsymbol{C}_N\boldsymbol{b}_N = \dfrac{1}{4}\boldsymbol{c}_1 + \dfrac{1}{12}\boldsymbol{c}_2 = \boldsymbol{1}.
\end{align*}
Thus, $\boldsymbol{b}_N$ is a solution to system \eqref{eq:lemma}, and because $\boldsymbol{C}_{N}$ is invertible, this solution also corresponds to the unique solution to \eqref{eq:lemma}.
\end{proof}
A direct consequence of Lemma \ref{lem:lem1} is the equality
\begin{align}\label{eq:expressionfor:Bj}
 \widetilde{B}_j^{(N)} = \sum_{k = 1}^N (\boldsymbol{C}_N^{-1})_{jk} = (\boldsymbol{C}_N^{-1}\boldsymbol{1})_j = b_j,\quad\text{for }j=1,\dots,N,
\end{align}
therefore $\widetilde{B}^{(N)}_j = 0$ for $j = 3,\dots,N$, for all $N>2$. In addition, we can immediately show that the vector
\newline $\widetilde{\boldsymbol{b}}_N = (\tfrac{1}{3}b_1,\tfrac{1}{5}b_2,0,0,\dots)^{\rm T}\in\mathbb{R}^N$
satisfies
\begin{align}
\widetilde{\boldsymbol{C}}_N^{\rm T}\widetilde{\boldsymbol{b}}_N = \boldsymbol{1}\quad\text{equivalently}\quad\widetilde{\boldsymbol{b}}_N =  (\widetilde{\boldsymbol{C}}_N^{\rm T})^{-1}\boldsymbol{1}=(\widetilde{\boldsymbol{C}}_N^{-1})^{\rm T}\boldsymbol{1}.\label{eq:expressionfor:Bj:tilde}
\end{align}
With the help of Lemma \ref{lem:lem1} we now prove Theorem \ref{thm:identical}.
\setcounter{theorem}{2}
\begin{theorem}[Identical closed RSWME systems]
     For \( N \geq 2 \), all RSWME systems are identical.
\end{theorem}
\begin{proof}
   To establish that the RSWME systems from \eqref{eq:Closed_form_RSWME_sys} are identical for all~$N \geq 2$, it suffices to demonstrate that the system matrix $A^N$ from \eqref{eq:Closed_form_RSWME} and the source term $S^N$ from \eqref{eq:sourcetermU} satisfy: $A^N(U) = A^2(U)$ and $S^N(U) = S^2(U)$ for all $N\geq 2$. Moreover, by their definitions, these equalities hold if and only if the following equalities are satisfied:
   \begin{align*}
    \Gamma^N = \Gamma^2,\quad \Phi^N = \Phi^2,\quad \Omega^N = \Omega^2,\quad \Lambda^N = \Lambda^2,\quad \text{for }N\geq 2,
   \end{align*}
where $\Gamma^N$, $\Phi^N$, $\Omega^N$ and $\Lambda^N$ are defined in \eqref{eq:GammaNPhiN} and \eqref{eq:OmegaNLambdaN}.
Using \eqref{eq:expressionfor:Bj}, and the fact that $\widetilde{B}_j^{(N)} = 0$ for all $j\geq 3$ and arbitrary $N>2$, we can directly conclude that
\begin{align*}
     \Gamma^N= \sum_{j=1}^N \frac{\left(\widetilde{B}^{(N)}_j\right)^2}{2j+1} = \frac{\left(\widetilde{B}_1^{(N)}\right)^2}{3} + \frac{\left(\widetilde{B}_2^{(N)}\right)^2}{5}=\frac{\left(\widetilde{B}_1^{(2)}\right)^2}{3} + \frac{\left(\widetilde{B}_2^{(2)}\right)^2}{5} = \Gamma^2,
\end{align*}
and 
\begin{align*}
    \Omega^N= \sum_{j=1}^N \widetilde{B}_j^{(N)}=\widetilde{B}_1^{(N)} + \widetilde{B}_2^{(N)}  = \widetilde{B}_1^{(2)} + \widetilde{B}_2^{(2)} = \Omega^2. 
\end{align*}
On the other hand, we define the matrix $\boldsymbol{P}_N = \widetilde{\boldsymbol{C}}_N^{-1}\boldsymbol{C}_N^{-1}$ where each entry is denoted as $(\boldsymbol{P}_N)_{ij} = P_{Nij}$ for all $i,j = 1,\dots,N$. Subsequently, by Lemma \ref{lem:lem1} and Equation \eqref{eq:expressionfor:Bj:tilde}, it follows that
\begin{align*}
 \Phi^N  &= \sum_{i=1}^N \sum_{j=1}^N P_{Nij} = \boldsymbol{1}^{\rm T}(\boldsymbol{P}_N\boldsymbol{1}) = \boldsymbol{1}^{\rm T}(\widetilde{\boldsymbol{C}}_N^{-1}(\boldsymbol{C}_N^{-1}\boldsymbol{1})\\
 &=
 \boldsymbol{1}^{\rm T}(\widetilde{\boldsymbol{C}}_N^{-1}\boldsymbol{b}_N) = \big((\widetilde{\boldsymbol{C}}_N^{-1})^{\rm T}\boldsymbol{1}\big)^{\rm T}\boldsymbol{b}_N = (\widetilde{\boldsymbol{b}}_N)^{\rm T}\boldsymbol{b}_N = \dfrac{1}{3}b_1^2 + \dfrac{1}{5}b_2^2 \\&=\boldsymbol{1}^{\rm T}(\widetilde{\boldsymbol{C}}_2^{-1}\boldsymbol{C}_2^{-1}\boldsymbol{1})= \Phi^2.
\end{align*}
Finally, expressing $\Lambda^N$ in terms of $\Phi^N$ and $\Omega^N$ yields
\begin{align*}
 \Lambda^N = -\Phi^N + (\Omega^N)^2 = -\Phi^2 + (\Omega^2)^2 = \Lambda^2,
\end{align*}
thereby concluding the proof.
\end{proof}


\begin{thebibliography}{00}

\bibitem{Arcas2011}
D.~Arcas, Y.~Wei, {Evaluation of velocity-related approximations in the nonlinear shallow water equations for the {K}uril {I}slands, 2006 tsunami event at {H}onolulu, {H}awaii}, \emph{Geophys.~Res.~Lett.} 38~(12) (2011).

\bibitem{Bauerle2025}
M.~Bauerle, A.~J. Christlieb, M.~Ding, J.~Huang, On the {R}otational {I}nvariance and {H}yperbolicity of {S}hallow {W}ater {M}oment {E}quations in {T}wo {D}imensions, \emph{SIAM J.~Math.~Anal.} 57~(1) (2025) 1039--1085.

\bibitem{Bobylev2005}
A.~V. Bobylev, {{I}nstabilities in the {C}hapman-{E}nskog {E}xpansion and {H}yperbolic {B}urnett {E}quations}, \emph{J.~Stat.~Phys.} 124~(2–4) (2005) 371–399.

\bibitem{Canestrelli2009}
A.~Canestrelli, A.~Siviglia, M.~Dumbser, E.~Toro, {Well-balanced high-order centred schemes for non-conservative hyperbolic systems. {A}pplications to shallow water equations with fixed and mobile bed}, \emph{Adv.~Water Resour.} 32~(6) (2009) 834–844.

\bibitem{CastroDaz2012}
M.~J. Castro~Díaz, E.~Fernández-Nieto, {A class of computationally fast first order finite volume solvers: Pvm methods}, \emph{SIAM J.~Sci.~Comput.} 34~(4) (2012) A2173–A2196.

\bibitem{chapman1990mathematical}
S.~Chapman, T.~Cowling, The {M}athematical Theory of {N}on-Uniform {G}ases, 3rd Edition, Cambridge University Press, Cambridge, 1990.

{\color{black}
\bibitem{Ciallela2026}
M. Ciallela, J. Koellermeier, {{High order global flux schemes for general steady state preservation of shallow water moment equations with non-conservative products}}, \emph{Comput.~Fluids} 305 (2026), 106887.
} 

{\color{black}
\bibitem{Ciallella2022}
M. Ciallella, L. Micalizzi, P. Öffner, D. Torlo, {{An arbitrary high order and positivity preserving method for the shallow water equations}}, \emph{Comput.~Fluids} 247 (2022), 105630.
}

{\color{black}
\bibitem{Decoene2009}
A.~Decoene, L.~Bonaventura,  and E.~Miglio, and F.~Saleri, {Asymptotic derivation of the section-averaged shallow water equations for natural river hydraulics}, \emph{Math. Models Methods Appl. Sci.} 19 (3), (2009) 387--417.
}

{\color{black}
\bibitem{Elgamal2021}
M. Elgamal, {{A Moment-Based Chezy Formula for Bed Shear Stress in Varied Flow}}, \emph{Water} 13 (2021), 1254.
}

{\color{black}
\bibitem{Ersing2025}
P. Ersing, S. Goldberg, A. R. Winters, {{Entropy stable hydrostatic reconstruction schemes for shallow water systems}}, \emph{J.~Comput.~Phys.} 527 (2025), 113802.
}

\bibitem{JosGarresDaz2021}
J.~{Garres-Díaz}, M.~{Castro Díaz}, J.~{Koellermeier}, T.~{Morales de Luna}, {{S}hallow {W}ater {M}oment {M}odels for {B}edload {T}ransport {P}roblems}, \emph{Commun.~Comput.~Phys.} 30~(3) (2021) 903–941.

{\color{black}
\bibitem{Gerbeau2001}
J.-F.~Gerbeau, B.~Perthame, {Derivation of viscous Saint-Venant system for laminar shallow water; Numerical validation}, \emph{Discrete Contin.~Dyn.~Syst.~Ser.~B} 1~(1) (2001) 89--102
}

\bibitem{Huang2022}
Q.~{H}uang, J.~{K}oellermeier, W.~{Y}ong, {{E}quilibrium stability analysis of hyperbolic shallow water moment equations}, \emph{Math.~Methods Appl.~Sci.} 45~(10) (2022) 6459–6480.

{\color{black}
\bibitem{Kern2009}
M. Kern, P. Bartelt, B. Sovilla, O. Buser, {{Measured shear rates in large dry and wet snow avalanches}}, \emph{J.~Glaciol.} 55 (2009), 327--338.
}

\bibitem{koellermeier2025energy}
J.~Koellermeier, A new {E}nergy {E}quation {D}erivation for the {S}hallow {W}ater {L}inearized {M}oment {E}quations, Universiteit Gent, Ghent, Belgium, 2025, to appear in \textit{Book of Abstracts of the 9th International Conference on Advanced Computational Methods in Engineering and Applied Mathematics (ACOMEN2025)}.

\bibitem{JulianKoellermeier2020}
J.~Koellermeier, M.~Rominger, {{A}nalysis and {N}umerical {S}imulation of {H}yperbolic {S}hallow {W}ater {M}oment {E}quations}, \emph{Commun.~Comput.Phys.} 28~(3) (2020) 1038–1084.

\bibitem{Koellermeier2022}
J.~Koellermeier, E.~Pimentel-García, {Steady states and well-balanced schemes for shallow water moment equations with topography}, \emph{Appl.~Math.~Comput.} 427 (2022) 127166.

\bibitem{Kowalski}
J.~Kowalski, M.~Torrilhon, {{M}oment {A}pproximations and {M}odel {C}ascades for {S}hallow {F}low}, \emph{Commun.~Comput.~Phys.} 25~(3) (2019) 669–702.

\bibitem{kremer2006methods}
G.~Kremer, The {M}ethods of {C}hapman-{E}nskog and {G}rad and {A}pplications, Technical Report RTO-EN-AVT-194, NATO Research and Technology Organisation (2006).

{\color{black}
\bibitem{Lajeunesse2010}
E. Lajeunesse, L. Malverti, F. Charru, {{Bed load transport in turbulent flow at the grain scale: Experiments and modeling}}, \emph{J.~Geophys.~Res.~Earth Surf.} 115 (2010), F04001.
}

{\color{black}
\bibitem{LeMeur2015}
 H.V.J. Le Meur, {Derivation of a viscous Boussinesq system for surface water waves}, 
\emph{Asymptot.~Anal.}, 94 (3-4) (2015) 309--345.
}

\bibitem{Liang2010}
D.~Liang, Evaluating shallow water assumptions in dam-break flows, Proceedings of the Institution of Civil Engineers - Water Management 163~(5) (2010) 227--237.

\bibitem{Liggett}
J.~A. Liggett, {C}ritical {D}epth, {V}elocity {P}rofiles, and {A}veraging, \emph{J.~Irrig.~Drain.~Eng.} 119~(2) (1993) 416--422.

\bibitem{Miller_2007}
R.~N. Miller, Numerical {M}odeling of {O}cean {C}irculation, Cambridge University Press, 2007, p. 35–86.

{\color{black}
\bibitem{Rodriguez2012}
J.M.~Rodríguez, R.~Taboada-Vázquez, {Derivation of a new asymptotic viscous shallow water model with dependence on depth}, \emph{Appl.~Math.~Comput.} 219~(7) (2012) {3292--3307}
}

{\color{black}
\bibitem{Sanvitale2016}
N. Sanvitale, E. T. Bowman, {{Using PIV to measure granular temperature in saturated unsteady polydisperse granular flows}}, \emph{Granul.~Matter} 18 (2016), 1--12.
}

\bibitem{Sanz-Ramos_Blade_Oller_Furdada_2023}
M.~Sanz-Ramos, E.~Blade, P.~Oller, G.~Furdada, Numerical modelling of dense snow avalanches with a well-balanced scheme based on the {2D} shallow water equations, \emph{J.~Glaciol.} 69~(278) (2023) 1646–1662.

{\color{black}
\bibitem{Schaefer2013}
M. Schaefer, L. Bugnion, {{Velocity profile variations in granular flows with changing boundary conditions: insights from experiments}}, \emph{Phys.~Fluids} 25 (2013), 063303.
}

\bibitem{Scholz2023}
U.~Scholz, J.~Kowalski, M.~Torrilhon, {Dispersion in {S}hallow {M}oment {E}quations}, \emph{Commun.~Appl.~Math.~Comput.} 6~(4) (2023) 2155–2195.

\bibitem{struchtrup2005macroscopic}
H.~{S}truchtrup, {M}acroscopic {T}ransport {E}quations for {R}arefied {G}as {F}lows, Springer-Verlag Berlin Heidelberg, 2005.

\bibitem{Verbiest2025}
R.~Verbiest, J.~Koellermeier,{ Capturing {V}ertical {I}nformation in {R}adially {S}ymmetric {F}low {U}sing {H}yperbolic {S}hallow water {M}oment {E}quations}, \emph{Commun.~Comput.~Phys.} 37~(3) (2025) 810–848.

\bibitem{Vreugdenhil1994}
C.~Vreugdenhil, {Equations}, in: Numerical {M}ethods for {S}hallow-{W}ater {F}low, Springer Netherlands, Dordrecht, 1994, pp. 15--46.

\bibitem{Wang2023}
D.~Wang, {E}xtensions and {S}imulations of {S}hallow {W}ater {E}quations, Master's thesis, National University of {S}ingapore, Singapore (2023).


\end{thebibliography}

\end{document}